\documentclass[12]{amsart}
\usepackage{amsmath,amssymb,hyperref}
\usepackage{graphics}
\usepackage[all]{xy}
\usepackage{geometry} 
\usepackage{array,multirow,makecell}
\newtheorem{thm}{Theorem}

\newtheorem{cor}[thm]{Corollary}
\newtheorem{prop}[thm]{Proposition}
\newtheorem{lem}[thm]{Lemma}

\newtheorem{remark}[thm]{Remark}

\newcommand{\SL}{\mathrm{SL}_2}

\newcommand{\Log}{\boldsymbol{\mathrm{Log}}} 
\newcommand{\Logres}{\boldsymbol{\mathrm{Log}^{\mathrm{res}}}} 
\newcommand{\Irrun}{\boldsymbol{\mathrm{Irr}^{\mathrm{un}}}} 
\newcommand{\Irram}{\boldsymbol{\mathrm{Irr}^{\mathrm{ram}}}}

\DeclareMathOperator{\tr}{tr}
\DeclareMathOperator{\ord}{ord}

\begin{document}

\author[K. Diarra]{Karamoko DIARRA}
\address{ DER de Math\'ematiques et d'informatique, FAST, Universit\'e des Sciences, des Techniques et des Technologies de Bamako, BP: E $3206$ Mali.}
\email{karamoko.diara2005@yahoo.fr}
 
\author[F. Loray]{Frank LORAY}
\address{ Univ Rennes, CNRS, IRMAR - UMR 6625, F-35000 Rennes, France.}
\email{frank.loray@univ-rennes1.fr}

\title[Algebraic solutions of irregular Garnier systems]{Classification of algebraic solutions of irregular Garnier systems.}

\date{\bf \today}
\dedicatory{In the memory of Tan Lei}
\thanks{We thank CNRS, Universit\'e de Rennes 1, Henri Lebesgue Center and ANR-16-CE40-0008 project ``{\it Foliage}'' for financial support.  We also thank Simons Fundation's project NLAGA who invited us two in Dakar, we started working on this subject there. We finally thank Hiroyuki Kawamuko and Yousuke Ohyama for helpfull discussions
on the subject.}

\begin{abstract} We prove that algebraic solutions of Garnier systems in the irregular case
are of two types. The classical ones come from isomonodromic deformations of linear equations 
with diagonal or dihedral differential Galois group; we give a complete list in the rank $N=2$ case
(two indeterminates).The pull-back ones come from deformations of coverings over a fixed degenerate
hypergeometric equation; we provide a complete list when the differential Galois group is $\mathrm{SL}_2(\mathbb C)$.
By the way, we have a complete list of algebraic solutions for the rank $N=2$ irregular Garnier systems.
\end{abstract}

\maketitle
\tableofcontents

\section{Introduction}

During the last ten years, much have been done about classification of special solutions of isomonodromy equations.
Recall how Painlev\'e and Garnier differential equations arise in the computation of monodromy-preserving 
deformations of linear 
equations over the Riemann sphere.  Consider the general $2^{\text{nd}}$ order linear differential equation 
\begin{equation}\label{eq:Gal2ndOrderGarnierReg}
u''+f(x)u'+g(x)u=0,\ \ \ \ (u'=\frac{du}{dx})
\end{equation}
$$\left\{\begin{matrix}
f(x)=&\frac{\theta_{N+1}}{x}+\frac{\theta_{N+2}}{x-1}+\sum_{i=1}^N\frac{\theta_i}{x-t_i}-\sum_{j=1}^N\frac{1}{x-q_j}\\
g(x)=&\frac{c_0}{x}+\frac{c_1}{x-1}-\sum_{i=1}^N\frac{H_i}{x-t_i}+\sum_{j=1}^N\frac{p_j}{x-q_j}
\end{matrix}\right.$$
with $2N+3$ regular-singular points distributed as follows:
\begin{itemize}
\item $N+3$ essential singular points $x=t_1,\ldots,t_N,0,1,\infty$ with exponents $\theta_i$, $i=1,\ldots,N+3$, and
\item $N$ apparent singular points $x=q_1,\ldots,q_N$ (with trivial local monodromy).
\end{itemize}
Coefficients $c_0,c_1$ and $H_i$ can be explicitely determined as rational functions of all other parameters
$t_i$'s, $q_i$'s, $p_i$'s, and $\theta_t$'s after imposing the following constraints
\begin{itemize}
\item the singular point at $x=\infty$ is regular-singular 
with exponent $\theta_{N+3}$, 
\item the singular points  $x=q_1,\ldots,q_N$ are apparent.
\end{itemize}
Then, it follows from the works of Fuchs, Garnier, Okamoto, Kimura that 
an analytic deformation 
$$t\to (p_1(t),\ldots,p_N(t),q_1(t),\ldots,q_N(t))\ \ \ \text{with}\ \ \ t=(t_1,\ldots,t_N),$$
of equation (\ref{eq:Gal2ndOrderGarnierReg})
is isomonodromic (i.e. with constant monodromy) if, 
and only if\footnote{The ``only if'' needs that exponents $\theta_t$'s are not integers.} 
all $\theta_t$'s are fixed and other parameters satisfy the Hamiltonian system
\begin{equation}\label{eq:HamiltonianGarnierSystem}
\frac{dq_j}{dt_i}=\frac{\partial H_i}{\partial p_j}\ \ \ \text{and}\ \ \ \frac{dp_j}{dt_i}=-\frac{\partial H_i}{\partial q_j}\ \ \ 
\forall i,j=1,\ldots,N.
\end{equation}
The system (\ref{eq:HamiltonianGarnierSystem}) reduces to the Painlev\'e VI equation for $N=1$, and to the Garnier system
for $N>1$. It is integrable, in the sense that it admits a local solution for each initial data. These local solutions are
expected to be very transcendental in general, and this has been proved by Umemura in the Painlev\'e case $N=1$:
for any choice of $\theta_i$'s, the general solution cannot be explicitely expressed in terms of solutions of linear
differential equations (of any order), non linear differential equations of order $1$, or algebraic functions.
However, for special choices of $\theta_i$'s, there are Riccati solutions or algebraic solutions.
The first ones, called ``classical'', have been classified by Watanabe; the later ones have been
classified mainly by Boalch, and by Lisovyy and Tykhyy (see \cite{Boalch,LisovyyTykhyy}),
after a long period of works by Hitchin, Dubrovin, Mazzocco, Andreev, Kitaev... 
\cite{Hitchin,DM0,MazzoccoPicard,Doran,AK,AK2,Kitaev,Kitaev2,Kitaev3,Kitaev4,BoalchIco,BoalchTetraOcta,Boalch6,BoalchHigherGenus,VK,VK2}. In the Garnier case $N>1$, we expect a similar feature; classical solutions have been classified by 
Okamoto and Kimura in \cite{OK}, and by 
Mazzocco in \cite{MazzoccoGarnier}, but the classification of algebraic solutions is still open. 
Following Cousin, and Heu \cite{Cousin,CousinHeu}, algebraic solutions are finite branch solutions and come from finite orbits
of the Mapping-Class-Group on character varieties, or equivalently representations on the total space of an algebraic deformation of the punctured curve, extending the monodromy representation of (\ref{eq:Gal2ndOrderGarnierReg}).
On the other hand, the result of Corlette and Simpson \cite{CS}
shows that such representations are of three different origins:
\begin{itemize}
\item degenerate representations, i.e. taking values into a finite, dihedral or reducible group,
\item factorization through a representation on a fixed curve,
\item arithmetic quotient of a polydisc.
\end{itemize}
The reader will find a more precise statement in \cite{CS}. Let us just mention in the first case
the works of Girand \cite{Girand} and Komyo \cite{Komyo} for deformations of equation (\ref{eq:Gal2ndOrderGarnierReg})
with dihedral monodromy, and Cousin and Moussard \cite{CousinMoussard} in the reducible case.
Deformations of equation (\ref{eq:Gal2ndOrderGarnierReg}) having a finite group are algebraic 
and provide an algebraic Garnier solution in a systematic way, but computations can be very tedious
as it has been in the works of Boalch for the Painlev\'e case $N=1$.

In the second case, solutions are said of ``pull-back type'': the deformation of equation (\ref{eq:Gal2ndOrderGarnierReg})
is given in this case by the pull-back of a fixed differential equation (or instance rigid, hypergeometric, i.e. $N=0$)
by a family of ramified covers $f_t:\mathbb P^1\to\mathbb P^1$ (see section \ref{sec:pullback}). This method
has been used by Doran, Andreev, Kitaev, Vidunas to construct Painlev\'e VI algebraic solutions. For Garnier
systems $N>1$, all pull-back solutions with non degenerate linear monodromy have been classified by the first author 
in \cite{Diarra1} (see Proposition \ref{Prop:DiarraPullBack}). However, in the last case, we do not know
how to bound the arithmetic data in order to be able to classify. This is not using this trichotomy that 
algebraic solutions of Painlev\'e VI equation were found, but by brute force, which seems out of reach 
even in the case $N=2$. Recently, Calligaris and Mazzocco \cite{CM} gave a partial classification 
by using confuence of poles in order to exploit the Painlev\'e classification \cite{LisovyyTykhyy}.

So far, we have only considered linear differential equations with regular-singular points,
leading to the Painlev\'e VI equation and Garnier systems. There is a similar approach
for linear differential equations with irregular-singular points leading for instance, in the case $N=1$,
to the other Painlev\'e equations (see section \ref{sec:IsomDef}). They can be deduced from 
the regular-singular case by confluence of poles, and whose solutions parametrize isomonodromic and 
iso-Stokes deformations of linear differential equations with $4$ poles counted with multiplicity.
In a similar way, we can define irregular Garnier systems and, in the case $N=2$, they are listed
in the papers of Kimura \cite{Kimura} and Kawamuko \cite{Kawamuko}. For general case $N>1$,
such integrable systems also exist, due to the work of Malgrange (see Heu \cite{Heu} in the ramified case),
and a general formula can be found in the work of Krichever \cite{Krichever}.
Again we expect the general solution to be very transcendental, but there are classical and algebraic solutions.
In the Painlev\'e case, a complete classification of these special solutions can be found in \cite{OO};
see section \ref{sec:AlgebraicPainleveI-V} for the list of algebraic solutions for Painlev\'e I to V equations.
For Garnier systems, classical solutions have been investigated by Suzuki in \cite{Suzuki}. 
For several formal types, Kawamuko and Suzuki listed rational/algebraic solutions in \cite{KawamukoRat,Suzuki,KawamukoAlg}. 
This is all what is known so far about algebraic solutions of irregular Garnier systems. 
Our main result is a complete classification, as well as a complete list in the case $N=2$.

Recently, the second author, together with Pereira and Touzet, proved an irregular version of Corlette-Simpson
Theorem in \cite{LPT}. An immediate consequence is that an algebraic solution of an irregular Garnier
system is of one of the two following types
\begin{itemize}
\item {\bf classical}: comes from the deformation of a rank $2$ differential system with diagonal 
or dihedral differential Galois group,
\item {\bf pull-back}: comes from the deformation obtained by pull-back of a fixed linear differential equation
by a family ramified covers.
\end{itemize}
(see Corollary \ref{cor:structure}).
The main result of the paper is the classification of solutions of pull-back type. Let us describe more precisely the construction.
We consider a fixed meromorphic linear differential
equation $\mathcal E_0$ on $\mathbb P^1$, which can be a two-by-two system, a second-order scalar equation, or a rank two
vector bundle with a connection. Then, we consider a family of ramified covers $(\phi_t:\mathbb P^1\to \mathbb P^1)_t$
and the family of pull-back $\mathcal E_t:=\phi_t^*\mathcal E_0$. Clearly, the deformation $t\mapsto \mathcal E_t$
is isomonodromic and isoStokes, and this gives rise to a partial solution of a (possibly irregular) Garnier system;
moreover, if the family $(\phi_t)_t$ is algebraic, we get an algebraic partial solution.
Here, partial means that the time variable is a function of $t$ which may not be dominant,
and it won't be for general $\mathcal E_0$ and $(\phi_t)_t$. When the dimension of deformation 
has the right dimension, namely $n-3$ where $n$ is the number of poles of $\mathcal E_t$
counted with multiplicity, then we get a complete algebraic solution. Sections \ref{sec:RamifiedCovers},
\ref{sec:Scattering}, \ref{sec:IrregEuler} and \ref{sec:Classifcover} are devoted to the classification 
of such solutions. Inspired by the similar classification in the logarithmic case established by the first author in \cite{Diarra1},
we define the irregular analogues of curve, Teichm\"uller and moduli spaces, Euler characteristic and Riemann-Hurwitz
formula. Then we prove that, assuming $\mathcal E_0$ irregular with differential Galois group not reduced to 
the diagonal or dihedral group (to avoid classical solutions), $\mathcal E_0$ is of degenerate hypergeometric type 
(at most $3$ poles counted with multiplicity) and the cover degree of $\phi_t$ is bounded by $6$.
Finally, the list of solutions is obtained by scattering poles to reduce to the list of \cite{Diarra1}.
In the pure Garnier case $N>1$ (i.e. excluding Painlev\'e equations) we obtain $3$ solutions, 
for Garnier systems of rank $N=2$ or $3$ (see Tables \ref{table:irregular} and \ref{table:irregularConfl}). 
They all come from pull-back of the degenerate (or ramified) Kummer equation $u''+\frac{2}{3x}u'-\frac{1}{x}u=0$
by coverings of degree $4$ and $6$.
Consequently:
{\it irregular Garnier systems of rank $N>3$ admit only classical algebraic solutions.}

In order to describe our classification result, let us introduce for each singular point of the linear differential
equation the following invariants:
\begin{itemize}
\item the Poincar\'e-Katz irregularity index $\kappa\in\frac{1}{2}\mathbb Z_{\ge0}$ which is such that,
after putting the linear differential equation $u''=g(x)u$ into Sturm-Liouville normal form,
the coefficient $g(x)$ has a pole of order $2\kappa+2$; 
\item the exponent $\theta\in\mathbb C$, defined up to a sign, which is the difference of eigenvalues
of the residue for the differential equation in matrix form when $\kappa\in\mathbb Z_{\ge0}$, and 
$\theta=0$ in the ramified case $\kappa\in\frac{1}{2}+\mathbb Z_{\ge0}$.
\end{itemize}
Denote by $\lceil\kappa\rceil$ the smallest integer satisfying $\kappa\le k$.
This formal data can be algebraically computed from the differential equation, and is invariant under isomonodromic/isoStokes deformations. In fact, after normalizing the linear differential equation by birational gauge transformation,
in order to minimize the number and order of poles, the differential equation is determined by 
\begin{itemize}
\item its irregular curve, i.e. the base curve equipped with local coordinates up to order $\lceil\kappa\rceil$ at each pole
(only the position for simple poles),
\item the monodromy data including Stokes matrices.
\end{itemize}
(see \cite{MRGalois,Krichever,BMM,vdPS}, and sections \ref{sec:IrregCurve} and \ref{sec:MonodStokes} for details).
With this in hand, for each (global) formal data (and fixed genus $g$)
$$\begin{pmatrix}\kappa_1&\cdots&\kappa_n\\
\theta_1&\cdots&\theta_n\end{pmatrix}\ \ \ \text{where}\ \ \ 
\left\{\begin{matrix}
\kappa_i\in\mathbb Z_{\ge0}&\Rightarrow&\theta_i\in\mathbb C\\
\kappa_i\in\frac{1}{2}+\mathbb Z_{\ge0}&\Rightarrow&\theta_i=0
\end{matrix}\right.$$
we get a quasi-projective moduli space of linear differential equations (see \cite{InabaSaito,Inaba}), 
and on this moduli space we get a polynomial foliation whose leaves
correspond to deformations of the differential equation (in fact of the spectral curve)
with constant monodromy data: we call it {\it isomonodromic foliation}. In the genus $g=0$ case, this is known as (degenerate) 
Garnier systems: we get a $N$ dimensional foliation on a $3N$-dimensional moduli space,
where $N=\sum_{i=1}^n \lceil\kappa_i\rceil+1$ is called the rank of the Garnier system. 
Leaves with algebraic closure correspond to algebraic solutions of the Garnier system.
For $N=1$, we find all Painlev\'e equations (see section \ref{sec:IsomDef}) and algebraic solutions,
in the irregular case ($\kappa_i>0$ for one $i$ at least), are listed in section \ref{sec:AlgebraicPainleveI-V}.

There are biregular isomorphisms between these foliated moduli spaces, due to the fact that
the normalization of a linear differential equation is not unique: after birational gauge transformation,
one can shift $\theta_i$'s by integers (and we can change their sign). 
We obviously classify algebraic solutions up to these isomorphisms.

\begin{thm}\label{thm:3nonclassical}
Up to isomorphisms, there are exactly $3$ non classical algebraic solutions, for irregular Garnier systems 
of rank $N>1$. The list of corresponding formal data is as follows:
$$\begin{pmatrix}0& 1&1\\ \frac{1}{3} & 0&1\end{pmatrix},\ \ \ \begin{pmatrix}1&2\\ 0 & 1\end{pmatrix}\ \ \ \text{and}\ \ \ 
\begin{pmatrix}1& 1&1\\ 0&0 & 1\end{pmatrix}.$$
\end{thm}

For the rank $N=2$ case, the list of explicit Garnier systems is provided in \cite{Kimura,Kawamuko}.
The two first algebraic solutions are as follows under Kimura's notations \cite{Kimura}:
\begin{itemize}
\item $H(1,2,2;2)$ with parameters $\varkappa_0=\varkappa_1=0$ and $\varkappa=\frac{2}{9}$ (and $\eta_0=\eta_1=1$):
$q_1$ is implicitely defined by\footnote{In Kimura's paper, 
canonical coordinates are denoted $\lambda_i$ and $\mu_i$ instead of variables $q_i$ and $p_i$ respectively, 
and Hamiltonians $K_i$ instead of $H_i$}
$$\left(\frac{q_1(q_1+1)}{(q_1-1)(q_1-2)}\right)^3=\left(\frac{t_2}{t_1}\right)^2,$$
and other variables are given by
$$q_2=\frac{q_1+1}{2q_1-1},\ \ \ p_1=-\frac{1}{2}\frac{t_1}{(q_1-1)^2}-\frac{1}{2}\frac{t_2}{q_1^2}-\frac{1}{6}\frac{2q_1-1}{q_1(q_1-1)}$$
$$\text{and}\ \ \ p_2=-\frac{1}{2}\frac{(2q_1-1)^2 t_1}{(q_1-2)^2}-\frac{1}{2}\frac{(2q_1-1)^2 t_2}{(q_1+1)^2}+\frac{1}{2}\frac{2q_1-1}{(q_1-2)(q_1+1)}.$$
\item $H(2,3;2)$ with parameters $\varkappa_0=0$ and $\varkappa_\infty=-\frac{1}{2}$ (and $\eta=1$):
$q_1$ is a solution of
$$\left(\frac{q_1(3q_1+2t_1)}{3}\right)^3=2t_2^2$$
and other variables are given by
$$q_2=-q_1-\frac{2}{3}t_1,\ \ \ p_1=\frac{q_1}{4}+\frac{t_1}{2}-\frac{1}{6q_1}-\frac{t_2}{2q_1^2}$$
$$\text{and}\ \ \ p_2=-\frac{q_1}{4}+\frac{t_1}{3}+\frac{1}{2(3q_1+2t_1)}-\frac{9t_2}{2(3q_1+2t_1)^2}.$$
\end{itemize}
In each case, the solutions $(t_1,t_2)\mapsto(p_1,p_2,q_1,q_2)$ satisfy the Hamiltonian system (\ref{eq:HamiltonianGarnierSystem})
for explicit Hamiltonians $H_i$ given in section \ref{sec:GarnierHamiltonians}.

The second solution coincides with one of the two solutions found by Kawamuko in \cite{KawamukoAlg}.
For the third solution, we are able to compute the algebraic isomonodromic deformation of the pull-back 
linear differential equation, but we don't know the explicit form of the Garnier system in that case.

\begin{thm}\label{thm:rank2classical}
Up to isomorphisms, classical algebraic solutions of irregular Garnier systems 
of rank $N=2$ occur exactly for the following formal data
\begin{itemize}
\item infinite discrete family:
$$\begin{pmatrix}0&0&0&\frac{1}{2}\\ \frac{1}{2}&\frac{1}{2}&\frac{1}{2}&0\end{pmatrix}$$
\item two-parameter families
$$\begin{pmatrix}0&0&0&1\\ \frac{1}{2}&\frac{1}{2}&\theta_1&\theta_2\end{pmatrix},\ \ \ 
\begin{pmatrix}0&\frac{1}{2}&0&0\\ \frac{1}{2}&0&\theta_1&\theta_2\end{pmatrix},\ \ \ 
\begin{pmatrix}0&0&0&1\\ 0&\theta_1&\theta_2&-\theta_1-\theta_2\end{pmatrix}$$
\item one-parameter families
$$\begin{pmatrix}0&0&2\\ \frac{1}{2}&\frac{1}{2}&\theta\end{pmatrix},\ \ 
\begin{pmatrix}0&\frac{1}{2}&1\\ \frac{1}{2}&0&\theta\end{pmatrix},\ \ 
\begin{pmatrix}\frac{1}{2}&\frac{1}{2}&0\\ 0&0&\theta\end{pmatrix},\ \
\begin{pmatrix}0&\frac{3}{2}&0\\ \frac{1}{2}&0&\theta\end{pmatrix},\ \ 
\begin{pmatrix}0&0&2\\ 0&\theta&-\theta\end{pmatrix},\ \ 
\begin{pmatrix}0&1&1\\ 0&\theta&-\theta\end{pmatrix}$$
\item sporadic solutions
$$\begin{pmatrix}\frac{1}{2}&\frac{3}{2}\\ 0&0\end{pmatrix},\ \ \ 
\begin{pmatrix}0&\frac{5}{2}\\ \frac{1}{2}&0\end{pmatrix}\ \ \ \text{and}\ \ \ 
\begin{pmatrix}0&3\\ 0&0\end{pmatrix}.$$
\end{itemize}
In the first case, there are countably many distinct algebraic solutions of unbounded
degree. In any other case, there is exactly one algebraic solution for each formal data.
\end{thm}
Kawamuko already discovered the fourth one-parameter family of solutions in \cite{KawamukoAlg}
and the third sporadic solution in \cite{KawamukoRat}. We provide the first sporadic solution 
in section \ref{sec:GarnierHamiltonians}.

\begin{cor}In the rank $N=2$ case, irregular Garnier systems with the following formal data: 
$$\begin{pmatrix}\frac{1}{2}&2\\ 0&\theta\end{pmatrix},\ \ \ 
\begin{pmatrix}1&\frac{3}{2}\\ \theta&0\end{pmatrix},\ \ \ 
\begin{pmatrix}4\\ \theta\end{pmatrix},\ \ \ 
\begin{pmatrix}\frac{7}{2}\\ 0\end{pmatrix},$$
have no algebraic solution.
\end{cor}

Sections \ref{sec:DiffEq}, \ref{sec:ConfluentHypergeometric} and \ref{sec:IsomDef} are folklore \cite{MR,MRGalois,Krichever,BMM,vdPS}. The Structure Theorem is presented in section \ref{sec:algsol}.
Sections \ref{sec:RamifiedCovers}, \ref{sec:Scattering}, \ref{sec:IrregEuler} and \ref{sec:Classifcover} are devoted 
to the classification of pull-back type solutions; in these sections, the irregular Euler characteristic is introduced and the 
irregular Riemann-Hurwitz formula is established. Classical solutions are classified in section \ref{sec:ClassicalN=2}
for the case $N=2$. Finally, explicit Hamiltonians are given in section \ref{sec:GarnierHamiltonians} for the above explicit
algebraic solutions.

\section{Linear differential equations}\label{sec:DiffEq}

In this paper, we consider rank 2 {\bf meromorphic connections} on curves.
This consists in the data of a rank 2 holomorphic vector bundle $E\to C$ on a complete smooth curve $C$,
together with a linear connection $\nabla:E\to E\otimes\Omega^1(D)$ where $D$ is the effective divisor of poles.
Precisely, $\nabla$ is a $\mathbb C$-linear map satisfying the Leibniz
rule $\nabla(f\cdot s)=df\otimes s+f\otimes\nabla(s)$ for any local function $f$ on $C$ and local section $s$ of $E$.
Locally on $C$, in trivializing coordinates for $E$, the connection writes
\begin{equation}\label{eq:generalsystem}
Y\mapsto\nabla(Y)=dY+A\cdot Y\ \ \ \text{with}\ \ \ A=\begin{pmatrix}\alpha&\beta\\ \gamma&\delta\end{pmatrix}
\ \ \ \text{and}\ \ \ Y=\begin{pmatrix}y_1\\ y_2\end{pmatrix}
\end{equation}
where $\alpha,\beta,\gamma,\delta$ are local sections of $\Omega^1(D)$ (meromorphic $1$-forms).
In general, the vector bundle $E$ is not trivial (globally) and we can only give such description in local charts on $C$.

\subsection{Bundle transformations}
There are two kinds of transformations we use to consider on connections. First of all, given a {\bf birational bundle transformation}
$\phi:E'\dashrightarrow E$ over $C$, we can define $\nabla':=\phi^*\nabla$ on $E'$ as the unique connection whose
horizontal sections are preimages by $\phi$ of $\nabla$-horizontal sections at a generic point of $C$.
Locally, $\phi$ is defined by $Y=M\cdot Y'$ with $M$ meromorphic, $\det(M)\not\equiv0$, and $\nabla'$
is defined by 
\begin{equation}\label{eq:gauge}
A'=M^{-1}AM+M^{-1}dM.
\end{equation}
Also, given a rank one meromorphic connection $(L,\zeta)$ over $C$, we can consider the {\bf twist}
$(E',\nabla')=(E\otimes L,\nabla\otimes\zeta)$ locally defined by $A'=A+\omega I$ where $\zeta=d+\omega$ 
and $I$ is the identity matrix. 
All these transformations can be equivalently considered locally, in the holomorphic/meromorphic setting.
Note that they can add, simplify or delete singular points of $\nabla$.
We simply call {\bf bundle transformation} the combination of these two kinds of transformations.
They can be used to trivialize the vector bundle, or also to minimize the support and order of poles.
In general, we cannot do this simultaneously, except on $C=\mathbb P^1$ (see Dekker's Theorem).
We will also use local/global biholomorphic/bimeromorphic bundle transformations depending
of the nature of $\zeta$ and $M$.

\subsection{Local formal data}
At the neighborhood of a singular point $x=0$ on $C$, 
up to bimeromorphic bundle transformation and change of coordinate $x\to\varphi(x)$, 
we are in one of the following models:
\begin{equation}\label{eq:ModelSing}
\begin{matrix}
\Log\hfill\hfill & A=\begin{pmatrix}\frac{\theta}{2}&0\\ 0&-\frac{\theta}{2}\end{pmatrix}\frac{dx}{x}, \hfill\hfill
& \theta\in\mathbb C\setminus\mathbb Z
& \text{logarithmic non resonant}\hfill\begin{pmatrix}0\\ \theta\end{pmatrix}\\
\Logres & A=\begin{pmatrix}\frac{n}{2}&x^n\\ 0&-\frac{n}{2}\end{pmatrix}\frac{dx}{x}, \hfill\hfill
& n\in\mathbb Z_{\ge0}
& \text{logarithmic resonant}\hfill\begin{pmatrix}0\\ n\end{pmatrix}\\
\Irrun\hfill\hfill & A=\begin{pmatrix}\frac{1}{2}&0\\ 0&-\frac{1}{2}\end{pmatrix}\left(\frac{dx}{x^{k+1}}+\theta\frac{dx}{x}\right)+\tilde A,
& \left\{\begin{matrix} k\in\mathbb Z_{>0}\\ \theta\in\mathbb C\end{matrix}\right.  
& \text{irregular unramified case}\hfill\begin{pmatrix}k\\ \theta\end{pmatrix}\\
\Irram & A=\begin{pmatrix}0&x\\ 1&0\end{pmatrix}\frac{dx}{x^{k+1}}+\tilde A,\hfill\hfill
& k\in\mathbb Z_{>0}  
& \text{irregular ramified case}\begin{pmatrix}k-\frac{1}{2}\\ 0\end{pmatrix}
\end{matrix}
\end{equation}
where $\tilde A$ is holomorphic. The matrix column on the right will be explained later.
The order of pole is minimal up to bimemorphic bundle transformation in all these models,
and it is therefore an invariant.
We call $\theta$ the {\bf exponent}; we set $\theta=n$ in the logarithmic non diagonal case $\Logres$, and
$\theta=0$ in the irregular ramified case $\Irram$.
In fact, only $\cos(2\pi\theta)$ really makes sense up to bimeromorphic bundle transformations, since 
$\theta$ can be shifted by integers under birational bundle transformations;
moreover, the variable permutation $y_1\leftrightarrow y_2$ in $Y$ changes the sign of $\theta$.

We now define the Katz {\bf irregularity index} $\kappa\in\frac{1}{2}\mathbb Z_{\ge0}$ by 
\begin{itemize}
\item $\kappa=0$ in the logarithmic case $\Log$ and $\Logres$,
\item $\kappa=k$ in the irregular unramified case $\Irrun$,
\item $\kappa=k-\frac{1}{2}$ in the irregular ramified case $\Irram$.
\end{itemize}
The main property of $\theta$ and $\kappa$ is that they are multiplicative under ramified cover:

\begin{prop}\label{Prop:ramificationindex}
If $\varphi(x)=x^n$, $n\in\mathbb Z_{>0}$, then $(\tilde E,\tilde\nabla)=\varphi^*(E,\nabla)$ has, 
up to bundle transformation, the following invariants
$$ \tilde \kappa=n\kappa\ \ \ \text{and}\ \ \ \tilde\theta=n\theta.$$
In particular, the class of irregular singular points is characterized by $\kappa\not=0$ and is stable
under ramified covers. 
\end{prop}

The proof is straightforward.

\begin{remark}\label{Rem:ramificationindex}In case of model $\Log$ with exponent $\theta=\frac{p}{q}$
rational, the pole becomes apparent after a ramification $\varphi(x)=x^n$ with $n=mq$ a multiple of $q$:
we get $\tilde\theta=mp\in\mathbb Z$ and the pole disappear after bundle transformation.
In a similar way, in case of model $\Irram$, when $n=2m$ is even, we get after bundle transformation
an unramified pole $\Irrun$ with $\tilde\theta=m$, that can be normalized to $\tilde\theta=0$ 
after an additional bundle transformation. 
\end{remark}

In irregular models $\Irrun$-$\Irram$, we can further
kill the holomorphic part $\tilde A$ by formal bundle transformation; however it is divergent in general.
This already shows that $\kappa$ and $\theta$ are the only formal invariant, i.e. that can be 
algebraically computed. In the sequel, we denote by $\bar\kappa:=k$ the smallest integer $\ge\kappa$, i.e.  $\bar\kappa=\kappa$ or $\kappa+\frac{1}{2}$. There are other invariants, called Stokes matrices (see below), 
whose computation is very transcendental.

To resume, each singular point is characterized, up to base change and formal bundle transformation, 
by its irregularity $\kappa\in\frac{1}{2}\mathbb Z_{\ge0}$, and $\theta\in\mathbb C$, its exponent. 
We will call {\bf local formal data} of a differential
equation the matrix 
\begin{equation}\label{eq:LocForData}
\begin{pmatrix}\kappa_1&\cdots&\kappa_n\\
\theta_1&\cdots&\theta_n\end{pmatrix}\ \ \ \text{where}\ \ \ 
\left\{\begin{matrix}
\kappa_i\in\mathbb Z_{\ge0}&\Rightarrow&\theta_i\in\mathbb C\\
\kappa_i\in-\frac{1}{2}+\mathbb Z_{>0}&\Rightarrow&\theta_i=0
\end{matrix}\right.
\end{equation}
specifying the formal type at each singular point. In (\ref{eq:ModelSing}), the formal type is indicated on the rightside.

\subsection{Normalization}

\begin{prop}\label{prop:normalizedequation}
Any (global) connection $(E,\nabla)$ is equivalent, up to birational bundle tranformation,
to a $\SL$-connection $(E_0,\nabla_0)$ which locally fits with one of the models 
$\Log$, $\Logres$, $\Irrun$ or $\Irram$
at any pole along the curve, for a convenient choice of coordinate $x$ and trivialization of $E$.
Such a reduction is not unique, and we can moreover assume, up to additional birational bundle transformation,
that formal data satisfies
$$0\le\Re(\theta_i)\le\frac{1}{2}\ \ \ \text{for}\ i=1,\ldots,n-1\ \ \ \text{and}\ \ \ 0\le\theta_n<1.$$
If $\kappa_i,\theta_i\not\in\frac{1}{2}\mathbb Z$, then this latter reduction is unique (up to permution of poles).
\end{prop}

We call {\bf normalized equation} a connection like $(E_0,\nabla_0)$ in the statement.

\begin{proof}The algorithm is as follows. We alternate birational bundle transformations and twists in order to simplify
the poles (minimize) and then apply a final twist to get the  $\SL$-form. Let us firstly dicuss the second step.
Given a connection $(E,\nabla)$, then its trace admits a square root
$$(\det(E),\tr(\nabla))=(L,\zeta)^{\otimes 2}$$
if, and only if, $\deg(E)$ is even. In that case, $(E,\nabla)\otimes (L,\zeta)^{\otimes(-1)}$ is in $\SL$-form.

The first step is done by applying
successive elementary transformations $E'\dashrightarrow E$ at each pole $p$, i.e. of the form 
$$Y=M\cdot Y'\ \ \ \text{with}\ \ \ M=\begin{pmatrix}1&0\\ 0&x\end{pmatrix}$$
in convenient local coordinate $x$ and trivializations $Y,Y'$. 
In a more intrinsic way, if $l\subset E\vert_p$ denotes the direction spanned by $Y=\begin{pmatrix}1\\ 0\end{pmatrix}$,
then $E'$ is defined as the locally free sheaf whose sections are those sections of $E$ which, in restriction to $E\vert_p$,
belong to the direction $l$. We note that $E$ and $E'$ are canonically isomorphic over the complement $C\setminus p$;
moreover, $\deg(E')=\deg(E)+1$.
Now, if $\nabla$ is a connection on $E$ with a pole at $p$, we say that the elementary transformation is $\nabla$-adapted
if $l$ is an eigendirection of the leading term of the matrix connection; it is equivalent to the fact that
the induced connection $\nabla'$ on $E'$ has a pole of order not greater than $\nabla$ at $p$:
$$
A=\begin{pmatrix}a(x)&b(x)\\ c(x)&d(x)\end{pmatrix}\frac{dx}{x^{k+1}}\ \ \ \Rightarrow\ \ \ 
A'=M^{-1}AM+M^{-1}dM=\begin{pmatrix}a(x)&xb(x)\\ \frac{c(x)}{x}&d(x)+x^k\end{pmatrix}\frac{dx}{x^{k+1}}.$$
We now proceed to simplify poles by applying adapted elementary transformations.
\begin{itemize}
\item If $A(0)$ is scalar, i.e. of the form $I\frac{dx}{x^{k+1}}$, then it can be killed by a twist, and the order $k$ decreases.
\item If $A(0)$ is semi-simple but not scalar, then we can reduce to model $\Log$, $\Logres$ or $\Irrun$ 
by biholomorphic bundle equivalence and change of coordinate
(with possibly $\theta\in\mathbb Z$ in $\Log$, in which case the pole can be deleted by birational bundle equivalence)
and passing to the $\SL$-form.
\item If $A(0)$ is not semi-simple, then we apply a $\nabla$-adapted elementary transformation;
if $\nabla$ and $\nabla'$ have same order at $p$, then one can check that we are in models $\Logres$ or 
$\Irram$ up to holomorphic bundle equivalence and change of coordinate. 
\end{itemize}
Finally, after finitely many elementary steps, we arrive at one of the models, up to biholomorphic bundle transformation.
At the end, if $\deg(E)$ is even, we obtain the $\SL$-form after a twist. If not, we can apply a $\nabla$-adapted 
elementary transformation at one of the poles to shift $\deg(E)$ by $+1$, so that it becomes even, and then normalize by a twist. 
The first part of the statement is proved. 

The lack of unicity comes from 
\begin{itemize}
\item the possibility of performing an even number of additional $\nabla$-adapted elementary transformation,
\item the possibility of changing the $\SL$-normalization by twisting with a $2$-torsion holomorphic connection
(this freedom does not occur on $C=\mathbb P^1$).
\end{itemize}
The first operation has the effect to shift exponents $\theta_i\mapsto\theta_i+n_i$, $n_i\in\mathbb Z$, with $\sum\theta_i\in 2\mathbb Z$
(except in case $\Irram$ where $\theta$ is always zero). This does not affect neither the type $\Log$, $\Logres$, $\Irrun$ or $\Irram$
of the pole, nor irregularity $\kappa$, but only $\theta$. Recall also that $\theta_i$'s are defined up to a sign. We promptly 
deduce that all $\theta_i$ can be normalized with $0\le\Re(\theta_i)\le\frac{1}{2}$ by birational bundle transformation:
first put $-\frac{1}{2}\le\Re(\theta_i)\le\frac{1}{2}$ by shifting it by integers, and then switch to $-\theta_i$ if necessary.
But since we need $\det(E)$ even for the final $\SL$-normalization, then we possibly need to apply an additional
elementary transformation, and then shift one of the $\theta_i$'s by one.
\end{proof}

\begin{remark}\label{rem:normalizedequation}
If one of the poles has formal data of one of the following types
$$\begin{pmatrix}\kappa\\ \theta\end{pmatrix}=\begin{pmatrix}0\\ \frac{1}{2}\end{pmatrix}\ \ \ \text{or}\ \ \ 
\begin{pmatrix}\frac{1}{2}\\ 0\end{pmatrix}$$
then one can furthermore assume that 
$$0\le\Re(\theta_i)\le\frac{1}{2}\ \ \ \text{for all}\ i=1,\ldots,n.$$
Indeed, it suffices to note that these two types of poles are invariant by some $\nabla$-adapted 
elementary transformation. So the last one needed to get $\det(E)$ even can be performed 
on this pole. 

Also note, in case $C=\mathbb P^1$ and no more than one pole is of the above type, that the normalization
is unique in that case.
\end{remark}

\subsection{Irregular curve}\label{sec:IrregCurve}
To encode the global formal structure of the differential equation, we have to take
into account the conformal type of the curve $C$ (when genus $g>0$), and the position
of singular points. 
But it is important to recall here that, to reach the models $\Irrun$-$\Irram$,
we really need a change of coordinate $x:=\varphi(x)$ in general, in order to normalize the principal part 
of the differential equation. In fact, it is enough to consider 
$\varphi$ polynomial of order $\bar\kappa$ (recall $\bar\kappa\in\mathbb Z_{\ge0}$ is $\kappa$ or $\kappa+\frac{1}{2}$);
in the logarithmic case $\kappa=0$, no change is needed. In other words, for irregular singular points,
there are $\bar\kappa$ other formal invariants when considering only bundle transformations, 
and they can be killed by base change. For this reason, the conformal type of the base curve $C$ should
be enriched with the additional data, at each irregular point $t_i$, of a $\bar\kappa_i$-jet of coordinate 
$x_i:(C,t_i)\to(\mathbb C,0)$ in which the 
equation can be reduced to the models $\Irrun$-$\Irram$. We call {\bf irregular curve}
the data $X:=(C,D,\{x_i\})$. The deformation space of the irregular curve locally identifies with $H^0((\Omega^1)^{\otimes 2}(D))$
and has dimension 
\begin{equation}\label{eq:dimirrTeich}
T:=3g-3+\deg(D)=3g-3+\sum_{i=1}^n(1+\bar\kappa_i).
\end{equation}
There is a global deformation space which is a principal bundle over the moduli space $M_{g,n}$ 
for the punctured curve $(C,\vert D\vert)$, whose fiber is the group product of diffeomorphism jets.
Once we know the irregular curve and local formal data, it remains to add some extra analytic invariants
given by the monodromy representation together with Stokes data.
They can only be computed algebraically from the differential equation, i.e. by means of above invariants, 
when $T\le3$. For larger $T$, one can deform the equation with fixed irregular curve and formal data 
and these invariants we are going to describe are transcendental functions of the coefficients of the equation.

\subsection{Monodromy and Stokes matrices}\label{sec:MonodStokes}
The {\bf monodromy representation} of the equation is a group morphism
$$\rho_\nabla:\pi_1(C\setminus\vert D\vert,t_0)\to\mathrm{GL}_2(\mathbb C)$$
defined as the monodromy of a local basis of solutions $B_0$.
More precisely, we can cover $C\setminus\vert D\vert$ (where $\vert D\vert$ is the support of $D$, i.e. set of poles)
by open sets $U_i$ over each of which the differential equation admits a basis of solutions $B_i$;
on overlappings $U_i\cap U_{j}$, we get $B_i=M_{i,j}B_{j}$ for a transition matrix $M_{i,j}\in\mathrm{GL}_2(\mathbb C)$.
The collection $(M_{i,j})$ defines an element $\rho_\nabla$ of 
$$H^1(C\setminus\vert D\vert,\mathrm{GL}_2(\mathbb C))\simeq\mathrm{Hom}(\pi_1(C\setminus\vert D\vert),\mathrm{GL}_2(\mathbb C))$$
(point of view of {\bf local systems}).
In fact, any loop $\gamma\in\pi_1(C\setminus\vert D\vert,x_0)$ based at $x_0$ can be covered, by compacity, 
by a finite number of such open sets $U_i$, where $\gamma$ crosses successively these open sets following 
the index order $i=0,1,\ldots,m$. Then, the analytic continuation of the initial basis of solutions $B_0$ writes:
$$\underbrace{B_0}_{\text{on}\ U_0}=\underbrace{M_{0,1}B_1}_{\text{on}\ U_1}=\underbrace{M_{0,1}M_{1,2}B_2}_{\text{on}\ U_2}=\cdots
=\underbrace{M_{0,1}M_{1,2}\cdots M_{m,0}B_0}_{\text{(back to $U_0$)}}=:B_0^\gamma,$$
$$\text{i.e.}\ \ \ B_0^\gamma=M^\gamma B_0\ \text{on}\  U_0,\ \ \ \text{with}\ M^\gamma:=M_{0,1}M_{1,2}\cdots M_{m,0}.$$
We then define the monodromy morphism by setting $\rho_\nabla(\gamma)= M^\gamma$. It depends on the choice 
of the basis $B_0$: it is well-defined up to conjugacy by an element $M\in\mathrm{GL}_2(\mathbb C)$:
$$B_0=M\tilde B_0\ \ \ \Rightarrow\ \ \ \tilde M^\gamma=M^{-1}M^\gamma M \ \ \ \Rightarrow\ \ \ \tilde\rho_\nabla=M^{-1}\rho_\nabla M.$$
For instance, in the first logarithmic model $\Log$, one easily calculate
$$B_0(x)=\begin{pmatrix}x^{-\frac{\theta}{2}}&0\\ 0&x^{\frac{\theta}{2}}\end{pmatrix}\ \ \ \text{and}\ \ \ M^\gamma=\begin{pmatrix}e^{i\pi\theta}&0\\ 0&e^{-i\pi\theta}\end{pmatrix}$$
where $\gamma(t)=e^{2i\pi t}$, $t\in[0,1]$.

For irregular singular points, it is more subtle since there are {\bf Stokes matrices}, playing the role of infinitesimal monodromy.
In the unramified case, the model $\Irrun$ with $\tilde A\equiv0$ has the fundamental basis and monodromy
$$B_0(x)=\begin{pmatrix}x^{-\frac{\theta}{2}}e^{\frac{1}{2kx^k}}&0\\ 0&x^{\frac{\theta}{2}}e^{-\frac{1}{2kx^k}}\end{pmatrix}\ \ \ \text{and}\ \ \ M^\gamma=\begin{pmatrix}e^{i\pi\theta}&0\\ 0&e^{-i\pi\theta}\end{pmatrix}$$
We call it the {\bf formal monodromy}. In the general case $\tilde A\not\equiv0$, 
equation $\Irrun$ cannot be reduced to the formal normal form $\tilde A\equiv0$
by holomorphic bundle transformation, but only by (generically divergent) formal bundle transformation;
we thus get a formal fundamental basis $\hat B(x)$ with monodromy $M^\gamma$. 
The same reduction can also be done holomorphically over each of the following $2\kappa=2k$ sectors
$$V_0:=\left\{x\ ;\ 0<\vert x\vert<r,\ \ \ \vert \arg(x)-\frac{\pi}{2k}\vert<\frac{\pi}{k}-\epsilon\right\}$$
(where $r,\epsilon>0$ are small enough) and
$$V_l:=\left\{x\ ;\ e^{-i\pi\frac{l}{k}}x\in V_0 \right\},\ l=0,\ldots,2k-1.$$
Note that these sectors are covering the punctured disc $\{0<\vert x\vert<r\}$. 
Going back to the initial model $\Irrun$, there exists a fundamental matrix $B_l$ over each sector $V_l$,
asymptotic to the formal one $\hat B(x)$.
Then, we have
$B_l=S_lB_{l+1}$ with $S_l\in\mathrm{GL}_2(\mathbb C)$ is unipotent, upper-triangular when $l$ is even,
and lower-triangular for $l$ odd. Therefore, the monodromy around $x=0$ canonically splits as
\begin{equation}\label{eq:StokesDecomp}
M^\gamma=\underbrace{
\begin{pmatrix}e^{i\pi\theta}&0\\ 0&e^{-i\pi\theta}\end{pmatrix}}_{\text{formal monodromy}}
\underbrace{\begin{pmatrix}1&s_1\\ 0&1\end{pmatrix}\begin{pmatrix}1&0\\ t_1&1\end{pmatrix}\cdots
\begin{pmatrix}1&s_k\\ 0&1\end{pmatrix}\begin{pmatrix}1&0\\ t_k&1\end{pmatrix}}_{\text{Stokes matrices}}
\end{equation}
and this {\bf Stokes decomposition} is unique up to conjugacy by a diagonal matrix (the choice of $\hat B(x)$), 
so that equivalent decompositions write
$$\begin{pmatrix}e^{i\pi\theta}&0\\ 0&e^{-i\pi\theta}\end{pmatrix}
\begin{pmatrix}1&cs_1\\ 0&1\end{pmatrix}\begin{pmatrix}1&0\\ c^{-1}t_1&1\end{pmatrix}\cdots
\begin{pmatrix}1&cs_k\\ 0&1\end{pmatrix}\begin{pmatrix}1&0\\ c^{-1}t_k&1\end{pmatrix},\ \ \ \text{for}\ c\in\mathbb C^*.$$
In fact, any two models $\Irrun$ with the same $k$ and $\lambda$
are equivalent by holomorphic bundle transformation if, and only if, their Stokes decomposition coincide up to diagonal conjugacy.
Moreover, any $2k$-uple $(s_1,\ldots,s_k,t_1,\ldots,t_k)$ is realisable as Stokes decomposition of a model $\Irrun$.
In particular, the differential equation is equivalent to the formal model by holomorphic bundle transformation
if, and only if, $s_l=t_l=0$ for $l=1,\ldots,k$.
In the {\bf ramified case}, we have a similar story and Stokes decomposition looks like 
\begin{equation}\label{eq:ramStokesDecomp}
M^\gamma=\underbrace{
\begin{pmatrix}0&1\\ -1&0\end{pmatrix}}_{\text{formal monodromy}}
\underbrace{\begin{pmatrix}1&s_1\\ 0&1\end{pmatrix}\begin{pmatrix}1&0\\ t_1&1\end{pmatrix}\cdots
\begin{pmatrix}1&s_k\\ 0&1\end{pmatrix}}_{\text{Stokes matrices}}
\end{equation}
where $k=\kappa+\frac{1}{2}$: there are $2\kappa$ matrices (like in the unramified case where irregularity is $\kappa=k$). 
This decomposition characterizes the analytic equivalence class of the differential 
equation, and any such Stokes data can be realized. We refer to \cite{MR,MRGalois,Krichever,BMM,vdPS} for more details.

\subsection{Differential Galois group}

The (differential) Galois group of a normalized differential equation $(C,E,\nabla)$ can be computed 
from the monodromy and Stokes data (see \cite{MRGalois}). In the logarithmic case, 
the Galois group is just the Zariski closure of the monodromy goup in $\SL(\mathbb C)$.
Recall that algebraic subgroups of $\SL(\mathbb C)$ are in the following list (up conjugacy):
\begin{equation}\label{eq:listGaloisGroup}
\begin{matrix}
\text{infinite} & 
\left\{\begin{matrix} C_\infty=\left\{\begin{pmatrix}\lambda&0\\ 0&\lambda^{-1}\end{pmatrix}\ ;\ \lambda\in\mathbb C^*\right\},\hfill\hfill\\
D_\infty=\left\{\begin{pmatrix}\lambda&0\\ 0&\lambda^{-1}\end{pmatrix},\begin{pmatrix}0&\lambda\\ -\lambda^{-1}&0\end{pmatrix},\ ;\ \lambda\in\mathbb C^*\right\}=\langle C_\infty,\begin{pmatrix}0&1\\-1&0\end{pmatrix}\rangle,\\
P_\infty=\left\{\begin{pmatrix}1&\mu\\ 0&1\end{pmatrix}\ ;\ \mu\in\mathbb C\right\},\hfill\hfill\\
T_\infty=\left\{\begin{pmatrix}\lambda&\mu\\ 0&\lambda^{-1}\end{pmatrix}\ ;\ \lambda\in\mathbb C^*,\ \mu\in\mathbb C\right\},\hfill\hfill\\
T_n=\left\{\begin{pmatrix}\lambda&\mu\\ 0&\lambda^{-1}\end{pmatrix}\ ;\ \lambda^n=1,\ \mu\in\mathbb C\right\},\hfill\hfill\\
\SL(\mathbb C)\hfill\hfill\end{matrix}\right.  \\
\text{finite} & 
\left\{\begin{matrix} C_n=\left\{\begin{pmatrix}\lambda&0\\ 0&\lambda^{-1}\end{pmatrix}\ ;\ \lambda^n=1\right\},\hfill \# n\\
D_n=\left\{\begin{pmatrix}\lambda&0\\ 0&\lambda^{-1}\end{pmatrix},\begin{pmatrix}0&\lambda\\ -\lambda^{-1}&0\end{pmatrix},\ ;\ \lambda^n=1\right\},\ \# 2n\\
\text{tetrahedral}\simeq A_4\ltimes\mathbb Z/2,\hfill \#24\\
\text{octahedral}\simeq S_4\ltimes\mathbb Z/2,\hfill \#48\\
\text{icosahedral}\simeq A_5\ltimes\mathbb Z/2,\hfill\#120
\end{matrix}\right.
\end{matrix}
\end{equation}
In the irregular case, we first define the local Galois group as
\begin{itemize}
\item at an unramified pole, it is generated in the same basis as (\ref{eq:StokesDecomp}) by the {\bf exponential torus}
$$C_\infty=\left\{\begin{pmatrix}\lambda&0\\ 0&\lambda^{-1}\end{pmatrix}\ ;\ \lambda\in\mathbb C^*\right\}$$
together with the Stokes matrices: we obtain $C_\infty$, $T_\infty$ or $\SL(\mathbb C)$
depending if all Stokes matrices are trivial, if one over two is trivial (i.e. all $s_i$'s or all $t_i$'s), or else.

\item at a ramified pole, it is generated in the same basis as (\ref{eq:ramStokesDecomp}) by the exponential torus, the permutation matrix
$$\begin{pmatrix}0&1\\ -1&0\end{pmatrix}$$
and the Stokes matrices: we obtain $D_\infty$ or $\SL(\mathbb C)$
depending if all Stokes matrices are trivial, or not. \end{itemize}
The global Galois group is the Zariski closure in $\SL(\mathbb C)$ of all local Galois groups and the global monodromy group. In particular, irregular differential equations have always $>0$ dimension due to the exponential torus:
we can only have $C_\infty$, $D_\infty$, $T_\infty$ or $\SL(\mathbb C)$.

\subsection{Link with scalar equations}
On $C=\mathbb P^1$, meromorphic connections were historically defined by 
higher order {\bf scalar differential equations}, and it is sometimes more convenient to work with them.
Consider the differential equation 
\begin{equation}\label{eq:scalar}
u''+f(x)u'+g(x)u=0,
\end{equation}
 where $f,g$ rational/meromorphic, and $u'=\frac{du}{dx}$.
Then, thinking of $(u(x_0),u'(x_0))$ as the space of initial conditions at a generic point $x_0$, 
it is natural to associate the {\bf companion system}
\begin{equation}\label{eq:companion}
\nabla=d+A\ \ \ \text{with}\ \ \ A=\begin{pmatrix}0&-g\\ 1&-f\end{pmatrix}dx.
\end{equation}
Indeed, identifying the standard basis with $(u(x),u'(x))$, then the connection satisfies 
$\nabla\cdot u=u'$ and $\nabla\cdot u'=u''=-fu'-gu$. In the projective coordinate $y=-y_1/y_2$, 
it induces the Riccati equation $y'+y^2+fy+g=0$, and we recover the initial scalar equation 
by setting $y=u'/u$.
Conversely, given a more general system (\ref{eq:generalsystem}), then we can first apply a twist to set $\alpha=0$
in the matrix $A$, and then use a gauge transformation of the form
\begin{equation}\label{eq:gaugecyclic}
M=\begin{pmatrix}1&G\\0&F\end{pmatrix},\ \ \ F\not\equiv0,
\end{equation}
(and a twist) to reduce the matrix $A$ in the companion form (\ref{eq:companion}): set $Fdx=\gamma$ and $G=0$.
One can further reduce (\ref{eq:companion}), or accordingly (\ref{eq:scalar}), by gauge transformation (\ref{eq:gaugecyclic})
with $F\equiv1$, or equivalently setting $u:=u/\exp(\int G)$; by this way, we can arrive to the unique SL-form 
\begin{equation}\label{eqSLscalar}
u''=\frac{s(x)}{2}u,\ \ \ s=f'+\frac{f^2}{2}-g
\end{equation}
also called ``Sturm-Liouville operator''. In fact, $s(x)$ is the {\bf Schwarzian derivative} of the quotient of any two 
independant solutions $u_1(x),u_2(x)$ of the initial equation (\ref{eq:scalar})
$$s(x):=\{\varphi,x\}=\left(\frac{\varphi''}{\varphi'}\right)'-\frac{1}{2}\left(\frac{\varphi''}{\varphi'}\right)^2,\ \ \ \varphi=\frac{u_1}{u_2}.$$
The above reduction to scalar equation depends on the choice of coordinate $Y$, or more precisely on the choice 
of the so called ``cyclic vector'' $\begin{pmatrix}1\\0\end{pmatrix}$. For a general rank 2 meromorphic connection $(E,\nabla)$,
the reduction to 2nd order scalar equation depend on the choice of a line subbundle $L\subset E$.
It is however important to notice that the irregularity index of an arbitrary connection $\nabla$ 
(not necessarily of the form $\Log$-$\Logres$-$\Irrun$-$\Irram$)
is directly given by the order of poles of any reduction to scalar equation, namely
$$1+\kappa=\max\{\ord(f),\ord(g)/2\}.$$

\section{Confluent hypergeometric equations}\label{sec:ConfluentHypergeometric}
On $C=\mathbb P^1$, when the polar locus $D$ has degree $3$, the connection can be determined, 
up to bundle equivalence, by its local formal data. After base change (i.e. applying
a Moebius transformation in $x$-variable), we can reduce to the following list 
of classical scalar equations.
$$
\xymatrix{
{\begin{matrix}
\begin{pmatrix}0&0&0\\ \theta_0&\theta_1&\theta_\infty\end{pmatrix}\\
\text{Hypergeometric}
\end{matrix}}
\ar[r]
&
{\begin{matrix}
\begin{pmatrix}0&1\\ \theta_0&\theta_\infty\end{pmatrix}\\
\text{Kummer}
\end{matrix}}
\ar[r]\ar[rd]
&
{\begin{matrix}
\begin{pmatrix}2\\ \theta_\infty\end{pmatrix}\\
\text{Weber}
\end{matrix}}
\ar[rd]
&\\
&&
{\begin{matrix}
\begin{pmatrix}0&\frac{1}{2}\\ \theta_0&0\end{pmatrix}
\end{matrix}}
\ar[r]
&
{\begin{matrix}
\begin{pmatrix}\frac{3}{2}\\ 0\end{pmatrix}\\
\text{Airy}
\end{matrix}}
}$$

The {\bf Gauss hypergeometric equation}
\begin{equation}\label{eq:Gauss}
u''+\left(\frac{(a+b+1)x-c}{x(x-1)}\right)u'+\frac{ab}{x(x-1)}u=0
\end{equation}
has $3$ simple poles at $x=0,1,\infty$ with respective exponents $\theta_0=c-1$, $\theta_1=a+b-c$ and $\theta_\infty=a-b$.

The {\bf Kummer equation} (also called ``confluent hypergeometric'')
\begin{equation}\label{eq:Kummer}
u''+\left(\frac{c}{x}-1\right)u'-\frac{a}{x}u=0
\end{equation}
has a logarithmic pole at $x=0$ with exponent $\theta_0=c$, and an irregular point at $x=\infty$ having 
irregularity index $\kappa=1$ and exponent $\theta_\infty=2a-c$. Its SL-form
$$u''=\left(\frac{1}{4}+\frac{2a-c}{2x}+\frac{(c-1)^2-1}{4x^2}\right)u$$
is also known as Whittaker equation. A particular case is Bessel equation when $\lambda=0$.
The monodromy data can be described as follows:
$$M_0\cdot M_\infty=I\ \ \ \text{with}\ \ \ M_\infty=\begin{pmatrix}e^{i\pi\theta_\infty}&0\\ 0&e^{-i\pi\theta_\infty}\end{pmatrix}
\begin{pmatrix}1&s\\ 0&1\end{pmatrix}\begin{pmatrix}1&0\\ t&1\end{pmatrix}$$
$$\text{and}\ \ \ \mathrm{trace}(M_0)=2\cos(\pi\theta_0)=\mathrm{trace}(M_\infty)=2\cos(\pi\theta_\infty)+e^{i\pi\theta_\infty}st.$$
In particular, one easily check that are equivalent:
$$\text{the Galois group reducible}\ \Leftrightarrow\ st=0\ \Leftrightarrow\ \mathrm{trace}(M_0)=\mathrm{trace}(M_\infty)$$
$$\ \Leftrightarrow\ \theta_\infty=\pm\theta_0\mod 2\mathbb Z
\ \Leftrightarrow\ a\in\mathbb Z\ \text{or}\ c-a\in\mathbb Z.$$
To determine when the differential equation is totally reducible (i.e. with diagonal monodromy),
we refer to \cite{DLR}. When irreducible, the Galois group of the normalized equation is $\SL(\mathbb C)$.

The {\bf Weber equation}
\begin{equation}\label{eq:Weber}
u''=(x^2-2a)u
\end{equation}
has a single irregular pole at $x=\infty$ with irregularity index $\kappa=2$ and exponent $\theta_\infty=2a-1$.
The monodromy data can be described as follows:
$$M_\infty=\begin{pmatrix}e^{i\pi\theta_\infty}&0\\ 0&e^{-i\pi\theta_\infty}\end{pmatrix}
\begin{pmatrix}1&s_1\\ 0&1\end{pmatrix}\begin{pmatrix}1&0\\ t_1&1\end{pmatrix}
\begin{pmatrix}1&s_2\\ 0&1\end{pmatrix}\begin{pmatrix}1&0\\ t_2&1\end{pmatrix}=I.$$
In particular, one easily check that are equivalent:
$$\text{the Galois group reducible}\ \Leftrightarrow\ \theta_\infty\in 2\mathbb Z
\ \Leftrightarrow\ a\in\frac{1}{2} +\mathbb Z.$$
When irreducible, the Galois group of the normalized equation is $\SL(\mathbb C)$.

The {\bf degenerate confluent hypergeometric equation}
\begin{equation}\label{eq:DC}
u''+\frac{c}{x}u'-\frac{1}{x}u=0
\end{equation}
has one logarithmic pole at $x=0$ with exponent $\theta_0=c$ and a ramified irregular point at $x=\infty$
with irregularity index $\kappa=\frac{1}{2}$. Its SL-form is the ``degenerate Whittaker'' equation
\begin{equation}\label{eq:DW}
u''=\left(\frac{1}{x}+\frac{(c-1)^2-1}{4x^2}\right)u
\end{equation}
that will be used in our computations. The monodromy data can be described as follows:
$$M_0\cdot M_\infty=I\ \ \ \text{with}\ \ \ M_\infty=\begin{pmatrix}0&1\\ -1&0\end{pmatrix}
\begin{pmatrix}1&s\\ 0&1\end{pmatrix}=\begin{pmatrix}0&1\\ -1&-s\end{pmatrix}$$
$$\text{and}\ \ \ s+2\cos(\pi\theta_0)=0.$$
The Galois group of the normalized equation is $\SL(\mathbb C)$, except when $s=0$
where the Galois group is dihedral. One easily check that the latter case holds if, and only if $c\in\frac{1}{2}+\mathbb Z$.

Finally, the most degenerate one, the {\bf Airy equation}
\begin{equation}\label{eq:Airy}
u''=xu
\end{equation}
has a single irregular pole at $x=\infty$ with irregularity index $\kappa=\frac{3}{2}$.
The monodromy data can be described as follows:
$$M_\infty=\begin{pmatrix}0&1\\ -1&0\end{pmatrix}
\begin{pmatrix}1&-1\\ 0&1\end{pmatrix}\begin{pmatrix}1&0\\ 1&1\end{pmatrix}\begin{pmatrix}1&-1\\ 0&1\end{pmatrix}=I$$
and the Galois group is $\SL(\mathbb C)$.

\section{Isomonodromic deformations}\label{sec:IsomDef}

When the degree of the polar divisor $N:=\deg(D)$ satisfies  $N>3$, then local formal data fail to determine 
the differential equation and we have non trivial deformations, even with constant monodromy
and Stokes data. We call them isomonodromic deformations.

To define monodromy data in family, we need to introduce the {\bf irregular Teichm\"uller} space, 
which is the moduli space of $(C,D,\{x_i\},\{\alpha_j,\beta_j,\gamma_i\})$
where we take into account a basis $(\alpha_j,\beta_j)$ for the fundamental group (with some base point $t_0$), 
as well as a loop $\gamma_i$ from the base point
to the singular point $t_i$, ending along $x_i\in\mathbb R>0$, for each irregular singular point.
This irregular Teichm\"uller space is described in \cite{Krichever,Heu}. Then we can define monodromy
representation in family, as well as Stokes matrices. This makes sense so talk about deformation with constant 
monodromy, that we call isomonodromic deformations.

Precisely, fix a genus $g$ and 
a local formal data (\ref{eq:LocForData}); denote by $T$ the corresponding irregular Teichm\"uller dimension
(\ref{eq:dimirrTeich}). Then, we can consider the moduli space of triples $(X,E,\nabla)$ where 
\begin{itemize}
\item $X=(C,D,\{x_i\})$ is an irregular curve of genus $g$, i.e. $C$ is a curve of genus $g$,
$$D=\sum_{i=1}^n(1+\bar\kappa_i)[t_i]\ \ \ \text{and}\ \ \ x_i\text{ is a local coordinate at }t_i;$$
\item $E$ is a rank $2$ vector bundle with trivial determinant $\det(E)=\mathcal O_C$;
\item $\nabla:E\to E\otimes\Omega^1_C(D)$ is a meromorphic trace-free connection with polar divisor $D$;
\item in local coordinate $x_i$, the connection is defined by $\Log$-$\Logres$-$\Irrun$-$\Irram$ with local formal data (\ref{eq:LocForData}).
\end{itemize}
We can consider $C$ belonging to the irregular Teichm\"uller space if we want to define the monodromy, or
belonging to the irregular moduli space of curves (or a smooth finite cover) for explicit computations.
The latter one is a quasi-projective variety of dimension $T=3g-3+\deg(D)$, and the former one 
is its universal cover, with Mapping-Class-Group acting as covering transformations.
Once the irregular curve $X$ is fixed, the moduli space $\mathcal M(X)$ of $(E,\nabla)$ with above constraints
is a quasi-projective variety of dimension $2T$, provided the local formal data (\ref{eq:LocForData})
is generic enough; in special cases, we have to consider semi-stable connections for some choice of 
weights, or for instance irreducible connections, in order to get such a nice moduli space. The total space 
$\mathcal M$ of connections with fixed genus $g$ and local formal data (\ref{eq:LocForData}) has therefore 
dimension $3T$. A point $(X,E,\nabla)$ on this moduli space is locally determined by an irregular curve 
and an irregular monodromy (representation + Stokes matrices).
If we fix the curve $X$ and deform the monodromy, we then get the moduli space $\mathcal M(X)$.
If we now fix the monodromy and deform the curve, then we get the so-called (universal) {\bf isomonodromic deformation}.
There is a $T$-dimensional foliation on the moduli space $\mathcal M$ whose leaves
correspond to maximal isomonodromic/isoStokes deformations. We call it {\bf isomonodromic foliation}.
Isomonodromic leaves are locally parametrized 
by the irregular Teichm\"uller space. When considering $X$ living in the moduli space of curves, 
the $3T$-dimensional space is algebraic and the isomonodromic foliation is expected to be 
a polynomial foliation (i.e. defined by polynomial differential equations). This fact is well-known 
when $C=\mathbb P^1$ and $D$ is reduced: isomonodromic differential equations are known as {\bf Garnier systems}
of rank $T$ in that case; the particular case $T=1$ leads to {\bf Painlev\'e equations} (where $\dot{q}$ denotes $\frac{dq}{dt}$):
\begin{equation}\label{eq:PainleveEquations}
\begin{matrix}
P_I\hfill :& \ddot{q}=6q^2+t\hfill\hfill\\
P_{II}(\alpha)\hfill :& \ddot{q}=2q^3+tq+\alpha\hfill\hfill\\
P_{III}(\alpha,\beta,\gamma,\delta)\hfill :& \ddot{q}=\frac{(\dot{q})^2}{q}-\frac{\dot{q}}{t}+\frac{\alpha q^2+\beta}{t}+\gamma q^3+\frac{\delta}{q}\hfill\hfill\\
P_{IV}(\alpha,\beta)\hfill :& \ddot{q}=\frac{(\dot{q})^2}{2q}+\frac{3}{2}q^3+4tq^2+2(t^2-\alpha)q+\frac{\beta}{q}\hfill\hfill\\
P_{V}(\alpha,\beta,\gamma,\delta)\hfill :& \ddot{q}=\left(\frac{1}{2q}+\frac{1}{q-1}\right)(\dot{q})^2-\frac{\dot{q}}{t}+\frac{(q-1)^2}{t^2}\left(\alpha q+\frac{\beta}{q}\right)
+\frac{\gamma q}{t}+\frac{\delta q(q+1)}{q-1}\\
P_{VI}(\alpha,\beta,\gamma,\delta)\hfill :& \ddot{q}=\frac{1}{2}\left(\frac{1}{q}+\frac{1}{q-1}+\frac{1}{q-t}\right)(\dot{q})^2-\left(\frac{1}{t}+\frac{1}{t-1}+\frac{1}{q-t}\right)\dot{q}\hfill\hfill\\
&\hfill+\frac{q(q-1)(q-t)}{t^2(t-1)^2}\left(\alpha+ \frac{\beta t}{q^2}+ \frac{\gamma (t-1)}{(q-1)^2}+ \frac{\delta t(t-1)}{(q-t)^2}\right)
\end{matrix}
\end{equation}
The corresponding isomonodromy equations are given by Table \ref{table:Painleve}.

\begin{table}[h]
\caption{Painlev\'e equations as isomonodromy equations}
\begin{center}
\begin{tabular}{|c|c|}
\hline 
Local formal data  & Isomonodromy equation \\
\hline
$\begin{pmatrix}0&0&0&0\\ \theta_0&\theta_1&\theta_t&\theta_\infty\end{pmatrix}$ & $P_{VI}\left(\frac{(\theta_\infty-1)^2}{2},-\frac{(\theta_0)^2}{2},\frac{(\theta_1)^2}{2},\frac{1-(\theta_t)^2}{2}\right)$\\
\hline
$\begin{pmatrix}0&1&0\\ \theta_0&\theta_1&\theta_\infty\end{pmatrix}$ & $P_{V}\left(\frac{(\theta_\infty)^2}{2},-\frac{(\theta_0+1)^2}{2},\theta_1,-\frac{1}{2}\right)$\\
\hline
$\begin{pmatrix}0&\frac{1}{2}&0\\ \theta_0&0&\theta_\infty\end{pmatrix}$ & $P_{V}\left(\frac{(\theta_\infty)^2}{2},-\frac{(\theta_0+1)^2}{2},-2,0\right)\sim $\\
&{\small $P_{III}\left(-4(\theta_0-\theta_\infty-1),-4(\theta_0+\theta_\infty),4,-4\right)$}\\
\hline
$\begin{pmatrix}0&2\\ \theta_0&\theta_\infty\end{pmatrix}$ & $P_{IV}\left(\theta_\infty,-2(\theta_0+1)^2\right)$\\
\hline
$\begin{pmatrix}0&\frac{3}{2}\\ \theta_0&0\end{pmatrix}$ & $P_{II}\left(\theta_0-\frac{1}{2}\right)$\\
\hline
$\begin{pmatrix}1&1\\ \theta_0&\theta_\infty\end{pmatrix}$ & $P_{III}\left(4\theta_\infty,-4\theta_0,4,-4\right)$\\
\hline
$\begin{pmatrix}1&\frac{1}{2}\\ \theta_0&0\end{pmatrix}$ & $P_{III}\left(-8,-4\theta_0,0,-4\right)$\\
\hline
$\begin{pmatrix}\frac{1}{2}&\frac{1}{2}\\ 0&0\end{pmatrix}$ & $P_{III}\left(4,-4,0,0\right)$\\
\hline
$\begin{pmatrix}3\\ \theta_\infty\end{pmatrix}$ & $P_{II}\left(\frac{1-\theta_\infty}{2}\right)$\\
\hline
$\begin{pmatrix}\frac{5}{2}\\ 0\end{pmatrix}$ & $P_{I}$\\
\hline
\end{tabular}
\end{center}
\label{table:Painleve}
\end{table}

Let us recall the case of Painlev\'e II equation $P_{II}(\alpha)$. The general linear differential equation
having a single pole at infinity with local formal data
$$\begin{pmatrix}3\\ 1-2\alpha\end{pmatrix}$$
can be normalized in scalar $\SL$-form (i.e. Sturm-Liouville operator)
\begin{equation}\label{eq:LinearP2}
u''=\left(x^4+tx^2+2\alpha x+2H_{II}+\frac{3}{4(x-q)^2}-\frac{p}{x-q}\right)u
\end{equation}
(coefficients of $x^4$ and $x^3$ have been normalized to $1$ and $0$ by an affine transformation in $x$-variable).
Here, $p,q$ are accessory parameters (i.e. initial conditions for the Painlev\'e equation) and 
\begin{equation}\label{eq:HamiltonianP2}
H_{II}=\frac{1}{2}\left(p^2-q^4-tq^2-2\alpha q\right)
\end{equation}
is determined so that the pole $x=q$ is apparent. After getting rid of apparent singular point, we can transform equation (\ref{eq:LinearP2}) into the Riccati equation 
\begin{equation}\label{eq:RiccatiP2}
y'+(x-q)y^2+(2x^2-2q^2+2p)y+(2p-2q^2-t)x+(1-2\alpha+2pq-2q^3-qt)=0
\end{equation}
or equivalently into the $\SL$-system (setting $y=y_1/y_2$)
\begin{equation}\label{eq:SystemP2}
\begin{pmatrix}y_1\\ y_2\end{pmatrix}'=\begin{pmatrix}-x^2+q^2-p & 2\alpha-1-2pq+2q^3+qt\\ x-q & x^2-q^2+p\end{pmatrix}\begin{pmatrix}y_1\\ y_2\end{pmatrix}.
\end{equation}
One can check, by direct computation, that the singularity at infinity (in  variable $z=1/x$) can be 
normalized by holomorphic gauge transformation to 
\begin{equation}\label{eq:SystemP2infty}
dY+AY=0\ \ \ \text{with}\ \ \ 
A=\begin{pmatrix}\frac{1}{2}&0\\ 0 & -\frac{1}{2}\end{pmatrix}\left(2\frac{dz}{z^4}+t\frac{dz}{z^2}+\left(1-2\alpha\right)\frac{dz}{z}+\text{holomorphic}\right).
\end{equation}
We retrieve, in the principal part, the fact that the first two coefficients have been normalized by affine transformation in $x$, 
the third coefficient stands for the time variable, and the fourth one, for the local formal data $\theta_\infty$.
A deformation $t\mapsto(p(t),q(t))$ of equation (\ref{eq:LinearP2})
is isomonodromic
(in fact, iso-Stokes in that case) if, and only if, it satisfies the Hamiltonian equations
\begin{equation}\label{eq:HamiltonianSystemP2}
\frac{dp}{dt}=-\frac{\partial H_{II}}{\partial q}\ \ \ \text{and}\ \ \ \frac{dq}{dt}=\frac{\partial H_{II}}{\partial p}.
\end{equation}
Equivalently, the corresponding curve in $(t,p,q)$-variables is in the kernel of the $2$-form 
\begin{equation}\label{eq:PoissonP2}
\omega=dp\wedge dq+dt\wedge dH_{II}.
\end{equation}
Using the second equation (\ref{eq:HamiltonianSystemP2}), we can express $p$ in terms of $q$ and $\dot{q}$;
substituting in the first equation, we deduce that $q(t)$ is a solution of Painlev\'e II equation.

In (\ref{eq:DegSchemePainleve}) we see the list of formal data of Painlev\'e type, and how they conflue to each other.

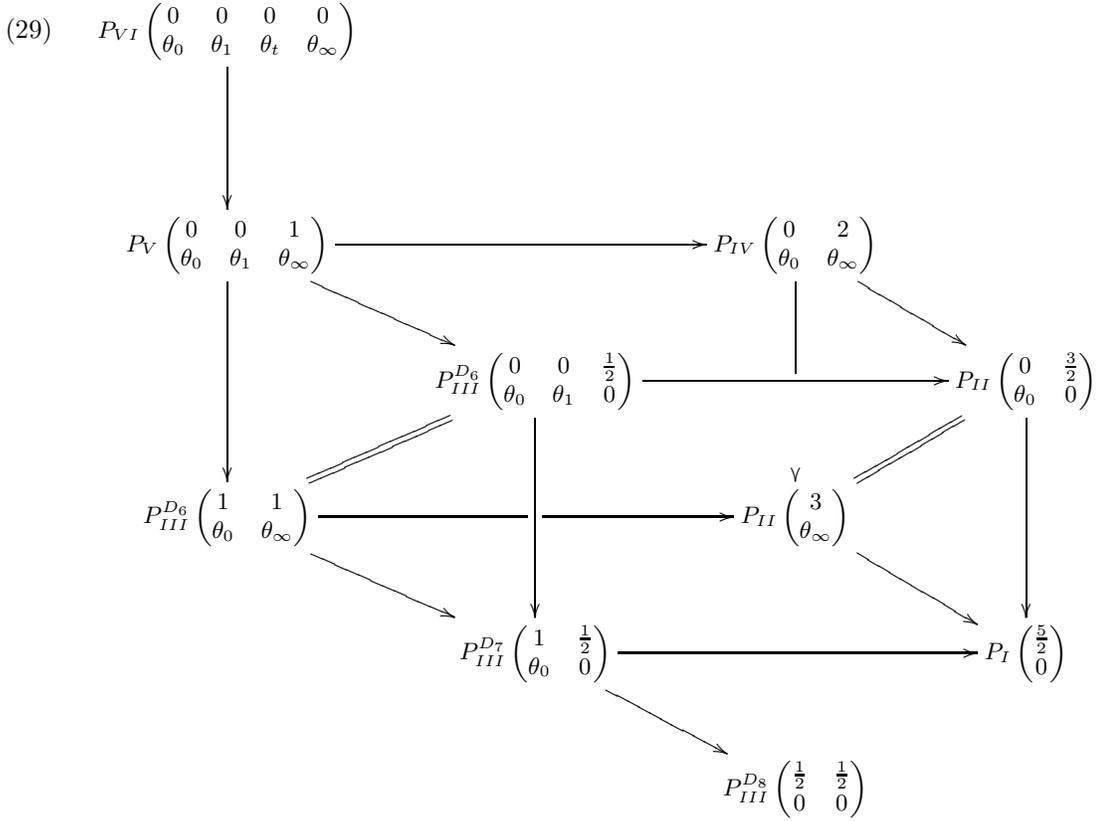
\begin{equation}\label{eq:DegSchemePainleve}
{\small
 \xymatrix {
 {P_{VI}\begin{pmatrix}0&0&0&0\\ \theta_0&\theta_1&\theta_t&\theta_\infty\end{pmatrix}}\ar[dd]  &&&\\
 &&&\\
  {P_V\begin{pmatrix}0&0&1\\ \theta_0&\theta_1&\theta_\infty\end{pmatrix}} \ar[rr] \ar[dd] \ar[dr] && {P_{IV}\begin{pmatrix}0&2\\ \theta_0&\theta_\infty\end{pmatrix}} \ar[dr] \ar[dd] |!{[dl];[dr]}\hole \\
  & {P_{III}^{D_6}\begin{pmatrix}0&0&\frac{1}{2}\\ \theta_0&\theta_1&0\end{pmatrix}} \ar[rr] \ar[dd] && {P_{II}\begin{pmatrix}0&\frac{3}{2}\\ \theta_0&0\end{pmatrix}} \ar[dd] \\
  {P_{III}^{D_6}\begin{pmatrix}1&1\\ \theta_0&\theta_\infty\end{pmatrix}}\ar@{=}[ur] \ar[rr] |!{[ur];[dr]}\hole \ar[dr] && {P_{II}\begin{pmatrix}3\\ \theta_\infty\end{pmatrix}}\ar@{=}[ur] \ar[rd]& \\
  & {P_{III}^{D_7}\begin{pmatrix}1&\frac{1}{2}\\ \theta_0&0\end{pmatrix}}\ar[dr] \ar[rr] && {P_I\begin{pmatrix}\frac{5}{2}\\ 0\end{pmatrix}} \\ 
    &&{P_{III}^{D_8}\begin{pmatrix}\frac{1}{2}&\frac{1}{2}\\ 0&0\end{pmatrix}}
  }}
\end{equation}

In the case $T=2$, the degeneration diagram is given in picture (\ref{eq:DegSchemeGarnier2}); 
the unramified part (and also the case $(7/2)$) of the list is treated by Kimura in \cite{Kimura}, 
while the ramified part is studied by Kawamuko in \cite{Kawamuko}. We named 
$\mathrm{Kim}_1,\ldots,\mathrm{Kim}_8, \mathrm{Kaw}_1,\ldots, \mathrm{Kaw}_8$
these equations following the order of appearance in these two papers.

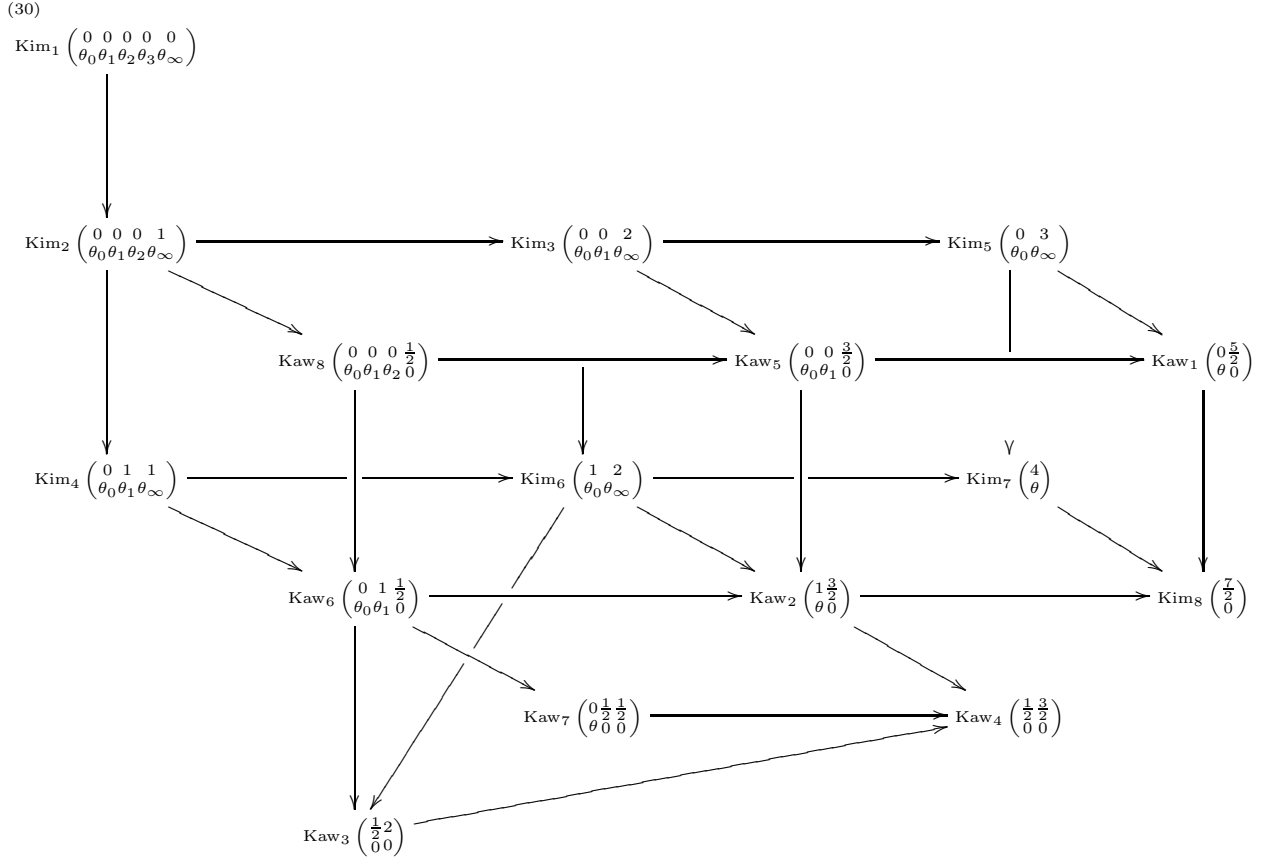
\begin{equation}\label{eq:DegSchemeGarnier2}
\tiny \xymatrix {
 {\mathrm{Kim}_1\left(\begin{matrix}0\\ \theta_0\end{matrix}\begin{matrix}0\\ \theta_1\end{matrix}\begin{matrix}0\\ \theta_2\end{matrix}\begin{matrix}0\\ \theta_3\end{matrix}\begin{matrix}0\\ \theta_\infty\end{matrix}\right)}\ar[dd]&&    &&    \\
 \\
 {\mathrm{Kim}_2\left(\begin{matrix}0\\ \theta_0\end{matrix}\begin{matrix}0\\ \theta_1\end{matrix}\begin{matrix}0\\ \theta_2\end{matrix}\begin{matrix}1\\ \theta_\infty\end{matrix}\right)} 
 \ar[rr] \ar[dd] \ar[dr] && 
 {\mathrm{Kim}_3\left(\begin{matrix}0\\ \theta_0\end{matrix}\begin{matrix}0\\ \theta_1\end{matrix}\begin{matrix}2\\ \theta_\infty\end{matrix}\right)} \ar[rr]\ar[dr] \ar[dd] |!{[dl];[dr]}\hole && 
 {\mathrm{Kim}_5\left(\begin{matrix}0\\ \theta_0\end{matrix}\begin{matrix}3\\ \theta_\infty\end{matrix}\right)} \ar[dr] \ar[dd] |!{[dl];[dr]}\hole\\
 & {\mathrm{Kaw}_8\left(\begin{matrix}0\\ \theta_0\end{matrix}\begin{matrix}0\\ \theta_1\end{matrix}\begin{matrix}0\\ \theta_2\end{matrix}\begin{matrix}\frac{1}{2}\\ 0\end{matrix}\right)} \ar[rr] \ar[dd] &&
 {\mathrm{Kaw}_5\left(\begin{matrix}0\\ \theta_0\end{matrix}\begin{matrix}0\\ \theta_1\end{matrix}\begin{matrix}\frac{3}{2}\\ 0\end{matrix}\right)} \ar[dd]\ar[rr] &&
 {\mathrm{Kaw}_1\left(\begin{matrix}0\\ \theta\end{matrix}\begin{matrix}\frac{5}{2}\\ 0\end{matrix}\right)} \ar[dd] \\
  {\mathrm{Kim}_4\left(\begin{matrix}0\\ \theta_0\end{matrix}\begin{matrix}1\\ \theta_1\end{matrix}\begin{matrix}1\\ \theta_\infty\end{matrix}\right)} \ar[rr] |!{[ur];[dr]}\hole \ar[dr] && 
  {\mathrm{Kim}_6\left(\begin{matrix}1\\ \theta_0\end{matrix}\begin{matrix}2\\ \theta_\infty\end{matrix}\right)} \ar[rd]\ar[rr] |!{[ur];[dr]}\hole \ar[lddd]|!{[dl];[dr]}\hole |!{[dl];[dd]}\hole&& 
  {\mathrm{Kim}_7\begin{pmatrix}4\\ \theta\end{pmatrix}}\ar[rd] \\
    & {\mathrm{Kaw}_6\left(\begin{matrix}0\\ \theta_0\end{matrix}\begin{matrix}1\\ \theta_1\end{matrix}\begin{matrix}\frac{1}{2}\\ 0\end{matrix}\right)}\ar[dr]\ar[dd] \ar[rr] && 
   {\mathrm{Kaw}_2\left(\begin{matrix}1\\ \theta\end{matrix}\begin{matrix}\frac{3}{2}\\ 0\end{matrix}\right)}
   \ar[rr]\ar[rd]&& 
    {\mathrm{Kim}_8\begin{pmatrix}\frac{7}{2}\\ 0\end{pmatrix}} \\
     &&
     {\mathrm{Kaw}_7\left(\begin{matrix}0\\ \theta\end{matrix}\begin{matrix}\frac{1}{2}\\ 0\end{matrix}\begin{matrix}\frac{1}{2}\\ 0\end{matrix}\right)}\ar[rr]&&
     {\mathrm{Kaw}_4\left(\begin{matrix}\frac{1}{2}\\ 0\end{matrix}\begin{matrix}\frac{3}{2}\\ 0\end{matrix}\right)}\\
    &
    {\mathrm{Kaw}_3\left(\begin{matrix}\frac{1}{2}\\ 0\end{matrix}\begin{matrix}2\\ 0\end{matrix}\right)}\ar[urrr]
  }
\end{equation}
It is important to notice that isomonodromy condition is equivalent to the fact that the deformation of equation
is induced by a flat meromorphic connection on the universal irregular curve, i.e. the total space of the family 
of irregular curve. 

\section{Algebraic solutions of irregular Garnier systems: examples and structure}\label{sec:algsol}

There are several methods to construct algebraic solutions of classical Garnier systems (logarithmic case),
see \cite{AK,Boalch,CM,Diarra1,Doran,DM0,LisovyyTykhyy} and references therein.
In the irregular case, let us describe two methods to produce algebraic isomonodromic deformations.

\subsection{Classical solutions: the Galois group is $C_\infty$ or $D_\infty$}\label{sec:classicalsol}
The rough idea is as follows. Since the differential Galois group of a linear differential equation can be determined
from its coefficients by algebraic operations, it follows that iso-Galois deformations are of algebraic nature:
we have an algebraic stratification of each moduli space 
$$\mathcal M_g\begin{pmatrix}\kappa_1&\cdots&\kappa_n\\
\theta_1&\cdots&\theta_n\end{pmatrix}$$  
as  defined in section \ref{sec:IsomDef} where strata are defined in term of the Galois group. In fact, 
this is not exactly true since there might be infinitely many strata corresponding to finite groups in the dihedral case.
However, the locus of each finite group is algebraic and coincides with a finite number of isomonodromic leaves.
It follows that the leaf associated to a finite linear group is algebraic. This has been extensively used in \cite{Hitchin,BoalchIco,BoalchTetraOcta,Boalch6,BoalchHigherGenus,Boalch}
and references therein in the logarithmic case; this does not occur in the irregular case since the Galois group is never finite
in that case. 
However, in a similar way, the locus of $C_\infty$ or $D_\infty$ is algebraic and is a finite union of isomonodromy leaves
in many cases when $C=\mathbb P^1$.

Let us first explain the diagonal case $C_\infty$. 
The poles can only be of type $\Log$ and $\Irrun$, and in this latter case, 
the Galois group $C_\infty$ coincides with all local exponential torii at irregular singular points
and all Stokes matrices are trivial. The two eigendirections of the Galois group correspond to 
two $\nabla$-invariant line bundles $L,L^{-1}\subset E$ of the vector bundle for the normalized equation,
and we have $E=L\oplus L^{-1}$.
The connection $\nabla$ restricts as meromorphic connections on $(L,\nabla\vert_L)$ and $(L,\nabla\vert_L)^{\otimes(-1)}$ 
and, at each pole $t_i$,
the corresponding residues are $\pm\frac{\theta_i}{2}$: they are opposite for $L$ and $L^{-1}$.
Fuchs relation yields:
$$\sum_{i=1}{n}\epsilon_k\theta_k\in\mathbb Z,\ \ \ \epsilon_i=\pm1.$$
There are finitely many such relations for each formal data, and given one relation,
the connection $(L,\nabla\vert_L)$ (and therefore $(E,\nabla)$) can be uniquely determined by Mittag-Leffler's Theorem
from the data of the irregular curve ($C=\mathbb P^1$ + principal parts). 

Something similar occur when the Galois group is $D_\infty$. The poles must be of type $\Log$, $\Irrun$ or $\Irram$
and the normal subgroup $C_\infty\subset D_\infty$ must coincide with all local exponential torii at irregular singular points.
In that case, Stokes matrices are trivial and, if $C=\mathbb P^1$, the global monodromy group is generated by matrices $M_i\in D_\infty$
(local monodromy at $t_i$) satisfying
$$M_1\cdots M_n=I.$$
Precisely, for a normalized equation, we have:
$$
\begin{matrix}
M_i\ \text{diagonal}\hfill\hfill  & \Rightarrow & \kappa_i\in\mathbb Z_{\ge0}\\
M_i\ \text{anti-diagonal} & \Rightarrow &
\begin{pmatrix}\kappa_i\\ \theta_i\end{pmatrix}=\begin{pmatrix}0\\ \frac{1}{2}\end{pmatrix}\ \text{or}\ \begin{pmatrix}\frac{1}{2}+n\\0\end{pmatrix},\ n\ge0.
\end{matrix}
$$
The number of anti-diagonal matrices among $M_1,\ldots,M_n$ is even.

\begin{remark}\label{rem:2antidiagDinfty}
One can check that, if $C=\mathbb P^1$ and there are only two anti-diagonal matrices among $M_1,\ldots,M_n$,
say $M_{n-1}$ and $M_n$, then again the (algebraic) locus of $D_\infty$ consists in a finite number of isomonodromic leaves.
Indeed, the monodromy representation is determined by the $M_i$'s; for $i=1,\ldots,n-2$, we have
$$M_i=\begin{pmatrix}\lambda_i&0\\ 0&\lambda_i^{-1}\end{pmatrix}\ \ \ \text{with}\ \lambda_i=e^{\pm \sqrt{-1}\pi\theta_i};$$ 
we can rescale $M_n=\begin{pmatrix}0&1\\-1&0\end{pmatrix}$ by diagonal conjugacy, and $M_{n-1}$ is determined by the relation.
The connection is also determined by the irregular curve in this case (principal parts).
\end{remark}

\subsection{Pull-back algebraic solutions}\label{sec:pullback}
Another way to construct algebraic isomonodromic deformations (see \cite{Doran,AK,Kitaev,AK2,Kitaev2,Kitaev3,Kitaev4,VK,VK2,Diarra1,OO}) is to fix a differential equation
$(C_0,E_0,\nabla_0)$ and consider an algebraic family $\phi_t:C_t\to C_0$ of ramified covers, 
where $t\in P$ a projective variety. The pull-back
$(C_t,\phi_t^*E_0,\phi_t^*\nabla_0)$ provides an algebraic isomonodromic deformation. Indeed, it is induced by 
the flat connection $(\Phi^*E_0,\Phi^*\nabla_0)$ over the total space $t:\mathcal C=\sqcup_t C_t\to P$ defined by pull-back
via the total ramified cover $\Phi:=(\phi_t,t):\mathcal C\to C_0\times P$. 

\subsection{Garnier algebraic solutions and apparent singular points}
So far, we have neglected to consider apparent singular points in isomonodromic deformations
since we were dealing with normalized equations. However, Garnier systems are derived from 
isomonodromic deformations of scalar differential equations (\ref{eq:scalar}) with $N$
poles (counted with multiplicity) and $N-3$ apparent singular points (see \cite{Kimura,Kawamuko}). 
This can be reinterpreted as the data of a connection $\nabla$ on the bundle 
$E=\mathcal O_{\mathbb P^1}\oplus\mathcal O_{\mathbb P^1}(-1)$ with $N$ poles, 
and we recover the scalar equation by taking $\mathcal O_{\mathbb P^1}\subset E$ as a cyclic vector.
Equivalently, and closer to our point of view, one can consider a $\SL$-connection $(E,\nabla)$
with $E=\mathcal O_{\mathbb P^1}\oplus\mathcal O_{\mathbb P^1}$, 
and the cyclic vector (with the right number of apparent singular points) is choosen to be
the constant line bundle $L\subset E$ fitting with one of the eigendirection over one pole (typically at $\infty$
when normalizing the position of poles on $\mathbb P^1$).
Here, the poles are of type $\Log$, $\Logres$, $\Irrun$ or $\Irram$, but in the case $\Log$, we also
allow $\theta\in\mathbb Z$ as a degenerate case of $\Logres$ when the singular point becomes apparent.
These considerations lead us to the following facts.

\begin{prop}\label{prop:cyclicvector}
Let $(E=\mathcal O_{\mathbb P^1}\oplus\mathcal O_{\mathbb P^1},\nabla)$ be an irregular $\SL$-connection, 
and $L\subset E$ be a cyclic vector like above.
The (local) isomonodromic deformation of $(E,\nabla)$ provides a (local) solution of the Garnier system
if the line bundle $L$ is not $\nabla$-invariant. If $\mathrm{Gal}(E,\nabla)=C_\infty$ and there is no apparent singular point,
then $L$ is $\nabla$-invariant and the deformation fails to provide a solution of the corresponding Garnier system.
\end{prop}

\begin{proof} The condition that $L$ is not $\nabla$-invariant is equivalent to the fact that it defines
a cyclic vector, which allow to define the scalar system with $N$ poles and $N-3$ apparent singular points.
This condition is preserved under isomonodromic deformations and the deformation of the scalar system
leads to a Garnier solution. In the case $(E,\nabla)$ is normalized (i.e. without apparent singular points)
and the Galois group is $C_\infty$, then $E=L_0\oplus L_0^{-1}$ with $L_0,L_0^{-1}$ two $\nabla$-invariant line bundles.
Automatically, we have $L_0\simeq\mathcal O_{\mathbb P^1}$ (a constant line bundle), and $L_0$ (resp. $L_0^{-1}$)
coincide with an eigendirection over each pole. It follows that $L_0$ or $L_0^{-1}$ coincides with $L$
everywhere, and is $L$ therefore $\nabla$-invariant.
\end{proof}

In the presence of {\bf apparent singular points}, the condition for a deformation to be isomonodromic
and giving rise to a Garnier solution is more subtle. 
Let $(E_0,\nabla_0)$ be a meromorphic $\SL$-connection, which is normalized except
at an apparent singular point $p_0\in C$: it is of type $\Log$ 
with $\theta=n\in\mathbb Z_{>0}$. Then after a birational bundle transformation $\phi:E_0\dashrightarrow E_0'$,
we can erase the singular point; moreover, we can assume $\phi$ supported by $p_0$, i.e. inducing a biholomorphic 
bundle tranformation outside of $p_0$. Generic local sections of $E_0$ at $p_0$ are transformed into sections of $E_0'$,
all tangent at the order $n-1$ to a given $\nabla$-invariant analytic subbundle $L_0\subset E_0'$. Then we have

\begin{prop}\label{prop:singappIsom}Under notations above, given a a deformation of $(E_0,\nabla_0)$ 
induced by a flat connection $(E,\nabla)$, are equivalent
\begin{itemize}
\item $(E,\nabla)$ is logarithmic near $p$,
\item the deformation $p$ of the singular point $p_0$ remains apparent 
and the corresponding deformation of line bundle $L_0\subset E_0'$ is induced by a local analytic $\nabla$-invariant 
line bundle $L\subset E'$.
\end{itemize}
In this case we say that the deformation is isomonodromic.
\end{prop}

Consequently, we will encode the data of such an equation $(E_0,\nabla_0)$ by the data of 
the normalized equation $(E_0',\nabla_0')$ together with the data of $p_0$ and the local 
$\nabla_0'$-invariant analytic line bundle$L_0$, or better by its fiber $l_0:=L_0\vert_{p_0}\subset E_0'\vert_{p_0}$
(initial condition) that is usually called {\bf parabolic data}. This method to deal with 
apparent singular points is used in \cite{IIS}.

\begin{remark}
The case $n=0$ also occur in solutions of Garnier systems as degenerate case
of singular points of type $\Logres$.
Garnier solutions can also be interpreted as before as deformation of a normalized equation $(E_0,\nabla_0)$
together with the parabolic data $(p_0\in C,l_0\in E_0\vert_{p_0})$ and there are as many Garnier solutions
as choices of initial condition $l_0$. 
\end{remark}

If we have several apparent singular points, the condition above must be imposed for each of them
to get an isomonodromic solution, and therefore giving rise to Garnier solutions.
By the way, we can have Garnier solutions corresponding to equation with Galois group $C_\infty$
in the presence of apparent singular points as we will see for Painlev\'e IV equation 
(second line of Table \ref{table:AlgSolPainleve}). 

\begin{prop}\label{prop:singappGarniersol}
The isomonodromic deformation of an irregular normalized connection $(E_0,\nabla_0)$
with an apparent singular point $(p_0,l_0)$ (notations above) is algebraic if and only if 
\begin{itemize}
\item the deformation of $(E_0,\nabla_0)$ is algebraic, with Galois group $C_\infty$ or $D_\infty$,
\item the local $\nabla_0$-invariant analytic line bundle $L_0$ defined by $l_0$ has algebraic closure,
i.e. corresponds to one of the two $\nabla_0$ invariant line bundles in diagonal case $C_\infty$,
or of the $2$-multivalued line bundle in dihedral case $D_\infty$.
\end{itemize}
\end{prop}

\begin{proof}The global deformation space of the irregular curve including the apparent singular point
may be viewed as a fiber bundle over the deformation space of the strict irregular curve, with fiber $C$
corresponding to the position of the apparent singular point for a fixed normalized equation.
Clearly, the deformation of $(E_0,\nabla_0)$ is algebraic if and only if it is algebraic in each of these two directions.  
But to be algebraic along $C$-fibers implies, due to Proposition \ref{prop:singappIsom}, that
the local analytic line bundle $L_0$ have algebraic closure. And this implies that the Galois group
must be $C_\infty$ or $D_\infty$ (the third possibility $\SL$ in the irregular case is excluded here).
\end{proof}

Proposition \ref{prop:singappGarniersol} can immediately be generalized to the case of several
apparent singular points.

\begin{remark}\label{rem:ApparentDihedral}
In the dihedral case $D_\infty$, an algebraic solution of a Garnier system with apparent singular point,
like in Proposition \ref{prop:singappGarniersol}, always arises in family, as limit of algebraic 
solutions with linear monodromy $D_\infty$ and no apparent singular points like in section \ref{sec:classicalsol}.
Indeed, from the monodromy side, this can be written as
$$M_1\cdots M_n=I\ \text{with}\ M_{n-2}=\begin{pmatrix}0&-\lambda\\ \lambda^{-1}&0\end{pmatrix},\ 
M_{n-1}=\begin{pmatrix}0&1\\ -1&0\end{pmatrix},\ 
\underbrace{M_{n}=\begin{pmatrix}1&0\\ 0&1\end{pmatrix}}_{\text{apparent}}$$
$$\rightsquigarrow
\tilde M_{n-2}=\begin{pmatrix}0&-\lambda_1\\ \lambda_1^{-1}&0\end{pmatrix},\ 
\tilde M_{n-1}=\begin{pmatrix}0&1\\ -1&0\end{pmatrix},\ 
\underbrace{\tilde M_{n}=\begin{pmatrix}\lambda_2&0\\ 0&\lambda_2^{-1}\end{pmatrix}}_{\text{non apparent}}$$
where $\lambda_1\lambda_2=\lambda$. For the differential equation, the apparent point is replaced by a non apparent
logarithmic singular point.
\end{remark}

\subsection{Structure theorem}

The key result for our classification is the:

\begin{thm}[\cite{LPT}]Let $(E,\nabla)$ be a flat meromorphic $\mathfrak{sl}_2$-connection 
on a projective manifold $X$. Then at least one of the following assertions holds true.
\begin{enumerate}
\item Maybe after passing to a (possibly ramified) two-fold cover $f: X' \to X$,
the connection $(E,\nabla)$ is equivalent to a diagonal connection on the trivial bundle:
$$\nabla=d+\begin{pmatrix}\omega&0\\ 0&-\omega\end{pmatrix}$$
with $\omega$ a rational closed $1$-form on $X$.
\item There exists a rational map $\Phi: X \dashrightarrow C$ to a curve and a meromorphic connection $(E_0,\nabla_0)$ on $C$
such that $(E,\nabla)$ is equivalent to $\Phi^*(E_0,\nabla_0)$ by bundle transformation.
\end{enumerate}
\end{thm}

A direct consequence is the

\begin{cor}\label{cor:structure}
Any algebraic solution of an irregular Garnier system comes from
\begin{enumerate}
\item either the deformation pull-back from a fixed connection $(C_0,E_0,\nabla_0)$ 
with $\mathrm{SL}_2(\mathbb C)$ Galois group via an algebraic family of ramified covers,
\item or the deformation of a connection having Galois group $D_\infty$,
\item or the deformation of a connection having Galois group $C_\infty$ with apparent 
singular point(s) like in Proposition \ref{prop:singappGarniersol}.
\end{enumerate}
\end{cor}

This result has been recently proved in the Painlev\'e case in \cite{OO} by checking {\it a posteriori}
the known list of algebraic solutions.
We expect to have a similar result for isomonodromy equations for curves of genus $g>0$.
However, it is still not known if these equations are polynomial in such a generality, so that the question
of algebraicity of solutions does not make sense so far.

To prove the Corollary, we just notice that any solution (algebraic of not) comes from an isomonodromic deformation,
or equivalently a flat connection over the total space $\mathcal C$ of the deformation curve. 
If the solution is algebraic, then $\mathcal C$
is algebraic, as well as the flat connection. We can therefore apply the above Theorem, and deduce the Corollary.

\subsection{Algebraic solutions of Painlev\'e I-V equations}\label{sec:AlgebraicPainleveI-V}

Table \ref{table:AlgSolPainleve} provides the list of algebraic solutions for irregular Painlev\'e equations
up to symmetries. 

\begin{table}[h]
\caption{Algebraic solutions of irregular Painlev\'e equations}
\begin{center}
\begin{tabular}{|c|c|c|c|c|c|c|}
\hline 
 Isomonodromy  & solution & name & Local formal   & Galois & pull-back ? & apparent\\
 equation &  & & data & group & & pole ? \\
\hline
$P_{V}\left(\frac{\theta^2}{2},-\frac{\theta^2}{2},0,-\frac{1}{2}\right)$ & 
$q=-1$ & $P_V$-rat & $\begin{pmatrix}0&1&0\\ \theta-1&0&\theta\end{pmatrix}$ & $\SL$ & yes & no \\
\hline
$P_{V}\left(\frac{\theta^2}{2},-\frac{1}{2},\theta,-\frac{1}{2}\right)$ &
$q=\frac{t}{\theta}+1$ & $P_V$-Lag & $\begin{pmatrix}0&1&0\\ 0&\theta&\theta\end{pmatrix}$ & $C_\infty$ & no & yes\\
\hline
$P_{V}\left(\frac{\theta^2}{2},-\frac{1}{8},-2,0\right)$ &
$q=2\frac{\sqrt{t}}{\theta}+1$ & $P_V$-alg & $\begin{pmatrix}0&\frac{1}{2}&0\\ \frac{1}{2}&0&\theta\end{pmatrix}$ & $D_\infty$ & no & no\\
\hline
 $P_{IV}\left(0,-\frac{2}{9}\right)$ & $q=-2t/3$ & $P_{IV}$-rat & $\begin{pmatrix}0&2\\ -\frac{2}{3}&0\end{pmatrix}$ 
 & $\SL$ & yes & no \\
\hline
 $P_{IV}\left(0,-2\right)$ & $q=-2t$ & $P_{IV}$-Her & $\begin{pmatrix}0&2\\ 0&0\end{pmatrix}$ & $C_\infty$ & yes & yes \\
\hline
 $P_{III}\left(4\theta,-4\theta,4,-4\right)$ & $q=\sqrt{t}$ & $P_{III}^{D_6}$-alg & $\begin{pmatrix}1&1\\ \theta&\theta\end{pmatrix}$ & $\SL$ & yes & no \\
\hline
 $P_{III}\left(4,-4,0,0\right)$ & $q=\sqrt{t}$ & $P_{III}^{D_8}$-alg & $\begin{pmatrix}\frac{1}{2}&\frac{1}{2}\\ 0&0\end{pmatrix}$ & $D_\infty$ & yes & no \\
\hline
 $P_{III}\left(-8,0,0,-4\right)$ & $q=\left(-\frac{t}{2}\right)^{1/3}$ & $P_{III}^{D_7}$-alg & $\begin{pmatrix}1&\frac{1}{2}\\ 0&0\end{pmatrix}$ & $\SL$ & yes & no \\
\hline
 $P_{II}\left(0\right)$ & $q=0$ & $P_{34}$-rat & $\begin{pmatrix}0&\frac{3}{2}\\ \frac{1}{2}&0\end{pmatrix}$ 
 & $D_\infty$ & yes & no \\
&& $P_{II}$-rat & $\begin{pmatrix}3\\ 1\end{pmatrix}$ & $\SL$ & yes & no \\
\hline
\end{tabular}
\end{center}
\label{table:AlgSolPainleve}
\end{table}

The name of solutions follows from \cite{OO}. Here ``rat'' and ``alg'' stand for ``rational'' and ``algebraic''
while ``Lag'' and ``Her'' stand for ``Laguerre'' and ``Hermite'' polynomials; also notation $P_{34}$ refers to Gambier's list.
The $P_V$ equation, when $\delta=0$, is equivalent to $P_{III}$ and solutions $P_{V}$-alg and $P_{III}^{D_6}$-alg
are therefore equivalent, but with different kind of linear local data, and this is why we keep the two solutions in the list.
For the same reason, we might consider $P_{34}$-rat and $P_{II}$-rat as different solutions.

\begin{remark}In table \ref{table:AlgSolPainleve}, apart pull-back solutions listed in tables
\ref{table:irregular} and \ref{table:irregularConfl}, there are two more pull-back solutions 
that correspond to dihedral Galois groups.  
\end{remark}

\section{Ramified covers and differential equations}\label{sec:RamifiedCovers}

Given a ramified cover $\phi:C\to C_0$ of degree $d$, and given a point $c_0\in C_0$, we can associate 
the pull-back divisor $\phi^*[c_0]=m_1[t_1]+\cdots+m_s[t_s]$ with $t_i\in C$ pair-wise distinct for $i=1,\ldots,s$.
This means that $\{t_1,\ldots,t_s\}$ is the fiber of $\phi$, and through convenient local coordinates $x_i$ near $t_i$, 
$\phi(x_i)=(x_i)^{m_i}$; we have $m_1+\cdots+m_s=d$, and for generic $t_0$, we have all $m_i=1$. 
We can therefore associate to $\phi$ its {\bf passport}
\begin{equation}\label{eq:passport}
\begin{pmatrix}\left[\begin{matrix}m_{1,1}\\ \vdots\\ m_{1,s_1}\end{matrix}\right] & \cdots &
\left[\begin{matrix}m_{\nu,1}\\ \vdots\\ m_{\nu,s_\nu}\end{matrix}\right]\end{pmatrix},\ \ \ \sum_{l=1}^{s_k}m_{k,l}=d\ \text{for}\ k=1,\ldots,\nu
\end{equation}
It is the data, for each critical value $c_1,\ldots,c_\nu\in C_0$ of $\phi$, of the corresponding partition
$m_{k,1}+\cdots+m_{k,s_k}=d$: there are $s_k$ points in the fiber $\phi^{-1}(c_k)$ with multiplicities $m_{k,l}$, $l=1,\ldots,s_k$.
The {\bf Riemann-Hurwitz} formula writes
\begin{equation}\label{eq:RiemHurw}
\underbrace{2-2g(C)}_{\chi(C)}=d(\underbrace{2-2g(C_0)}_{\chi(C_0)})-R\ \ \ \text{where}\ \ \ R=\text{total ramification}:=
\sum_{k=1}^n(d-s_k).
\end{equation}
Given a differential equation $(E_0,\nabla_0)$ with polar divisor $D$ on $C_0$, we can consider the pull-back
$(E,\nabla):=\phi^*(E_0,\nabla_0)$. It is easy to deduce the local formal data of $(E,\nabla)$ from that one
of $(E_0,\nabla_0)$ and the passport. Precisely, let $p\in C$ be a point of multiplicity $m$ for  $\phi$ (ramification $r=m-1$);
then the local formal datas of $(E_0,\nabla_0)$ at $\phi(p)$ and $(E,\nabla)$ at $p$ are related by:
$$(\kappa,\theta)=(m\kappa_0,m\theta_0)$$
{\bf except} when $\kappa=0$, $\theta_0\in\mathbb Q\setminus\mathbb Z$ and $m\theta_0\in\mathbb Z$ 
where the singular point becomes apparent, i.e. can be deleted by bundle transformation; in that latter case,
we indeed delete the pole. We deduce, for the respective order of poles, that 
$$\ord_p(\nabla)\le m\cdot \ord_{\phi(p)}(\nabla_0)+1-m$$
with strict inequality in the special case above, and ramified case $\kappa_0\not\in\mathbb Z$ and $m>1$.
On the other hand, when $\phi(p)$ is not a pole for $\nabla_0$, then $p$ is also non singular for $\nabla$.

In the sequel, we still fix the differential equation $(C_0,E_0,\nabla_0)$ and we want to deform the ramified cover 
and pull-back equation
$$\phi_t:C_t\to C_0,\ \ \ (E_t,\nabla_t):=\phi_t^*(E_0,\nabla_0).$$
We consider an irreducible algebraic family, parametrized by say $P\ni t$ (irreducible and projective), 
and there is a Zariski open subset $U\subset P$
where the passport of $\phi_t$ is locally constant, as well as the number $B$ of its critical values outside 
of the (fixed) polar locus of $\nabla_0$. Note that $B$ bounds the dimension of $P$, and maybe switching
to a larger family, i.e. with a larger parameter space $P$, we can assume $B=\dim(P)$.
The deformation $t\mapsto(C_t,E_t,\nabla_t)$, being automatically isomonodromic, locally factors through 
the universal isomonodromic deformation (see \cite{Heu}). In general, it defines a projective subvariety contained in the
transcendental isomonodromy leaf $\mathcal L$ of the corresponding isomonodromy foliation $\mathcal F$. 
But if $P\to\mathcal L$ is locally dominant, then it is globally dominant and the entire leaf $\mathcal L$ itself is algebraic,
giving rise to an algebraic solution of the corresponding isomonodromy equation. 
A necessary condition for this is that the dimension $T$ of $\mathcal L$ is bounded by the dimension $B$ of $P$.
We will call {\bf admissible} the data of a differential equation $(C_0,E_0,\nabla_0)$ and a passport (\ref{eq:passport})
such that $B\ge T$. For the sequel, it is convenient to add in the passport all trivial fibers 
$$\underbrace{1+\cdots+1}_{d\ \text{times}}$$
appearing over poles so that we may assume that $\nu=n+B$ in (\ref{eq:passport}) with entries $k=1,\ldots,n$
corresponding to fibers over the poles of $\nabla_0$, and entries $k=n+1,\ldots,\nu$ corresponding to free critical points
that are deformed along the family. Our aim now is to classify those admissible data such that the pull-back differential 
equation has irregular Teichm\"uller dimension $T\le B$. We will see in the next section that this inequality gives 
very strong constraints.

\section{Scattering ramifications}\label{sec:Scattering}

Suppose we are given a normalized differential equation $(C_0,E_0,\nabla_0)$ with poles $p_1,\ldots,p_n\in C_0$, and local formal data 
$(\kappa_i,\theta_i)_{i=1,\ldots,n}$ like (\ref{eq:LocForData}). Suppose we are given 
a ramified covering $\phi:C\to C_0$ with passport $(d=m_{k,1}+\cdots+m_{k,s_k})_{k=1,\ldots,\nu}$ over $c_1,\ldots,c_\nu\in C_0$ like (\ref{eq:passport})
where
\begin{itemize}
\item $(m_{k,l})_l$ for $k=1,\cdots,n$ correspond to the multiplicities of $\phi$ along fibers over the poles of $\nabla_0$ ($c_k=p_k$), some of them being possibly unbranched, i.e. $m_{k,l}=1$ for all $l$;
\item $(m_{k,l})_l$ for $k=n+1,\ldots,n+b$ correspond to fibers of $\phi$ over non singular points of $\nabla_0$,
all of which are branching, i.e. $m_{k,l}>1$ for at least one $l$.
\end{itemize}
Denote by $R_k:=d-s_k$ the total ramification number of the fiber over $c_k$. We denote by $(C,E,\nabla)$ the normalized
equation that can be deduced from the pull-back $\phi^*(E_0,\nabla_0)$ by bundle transformation. 
Let $N_k$ denote the number of poles counted with multiplicity in the fiber over $c_k$. 
The irregular Teichm\"uller dimension of $(C,E,\nabla)$ is given by 
$$T=3g-3+N=3g-3+N_1+\cdots+N_n$$
where $g$ is the genus of $C$, given by Riemann-Hurwitz Formula (\ref{eq:RiemHurw}) with $R:=\sum_{k=1}^\nu R_k$.
We now explain how to compute $N_k$ by means of $(\kappa_k,\theta_k)$ and $(m_{k,l})_l$.
\begin{itemize}
\item If $\kappa_k=0$ and $\theta_k\not\in\mathbb Q\setminus\mathbb Z$, then 
$$N_k=s_k=d-R_k.$$
\item If $\kappa_k=0$ and $\theta_k$ has order $m>1$ modulo $\mathbb Z$, then 
$$N_k=d-R_k-\#\{l=1,\ldots,s_k\ \vert\ m\ \text{divides}\ m_{k,l}\}.$$
\item If $\kappa_k\in\mathbb Z_{>0}$, then 
$$N_k=\sum_{l=1}^{s_k}(m_{k,l}+1)=d(\kappa_k+1)-R_k.$$
\item If $\kappa_k\in\mathbb Z_{>0}-\frac{1}{2}$, then
$$N_k=d(\kappa_k+1)-R_k+\#\{l=1,\ldots,s_k\ \vert\ 2\ \text{does not divide}\ m_{k,l}\}.$$
\end{itemize}
We assume $\phi$ admissible, i.e. that $T\le B$.

In this section, we show that we can replace $\phi$ by another ramified cover $\phi'$ with more 
critical points but less multiplicity in fibers in such a way that $T-B$ can only decrease. The total ramification
will be unchanged, but will be scattered oustide of the polar locus. This will
allow us to replace the deformation of $\phi$ by the wider deformation of $\phi'$, so that we will
be able to recover $\phi$ (and its deformation) by confluence of critical values. By the way,
the passport of $\phi'$ will be as simple as possible and it will be easy to classify such covers.

The first step consists in scattering the branching points over critical values outside of the polar locus. 
We call {\bf simple branching} a fiber of the form 
$$\left[\begin{matrix}m_{k,1}\\ m_{k,2}\\ \vdots\\ m_{k,s_k}\end{matrix}\right]=\left[\begin{matrix}2\\ 1\\ \vdots\\ 1\end{matrix}\right]$$
i.e. with $R_k=1$.

\begin{lem}\label{lem:FibreCritique}
Let $n<k\le\nu$, i.e. $c_k$ not a pole of $\nabla$. 
We can deform $\phi\rightsquigarrow\phi'$ over a neighborhood of $c_k$ so that
the single fiber of $\phi$ with total ramification $R_k$ is replaced by $R_k$ simple branching fibers for $\phi'$.
We have increased $B$ without changing $R$ or $T$.
\end{lem}

\begin{proof}Fix a disc $\Delta\subset C_0$ in which $c_k$ is the unique critical value of $\phi$.
The monodromy of the ramified covering is given by the permutation that decomposes into 
the product of $s_k$ cyclic permutations of orders $m_{k,1},\ldots,m_{k,s_k}$ with disjoint support.
In order to construct $\phi'$, it is enough to define its monodromy in $\Delta$, 
namely the data of $R_k$ transpositions whose product is the monodromy of $\phi$.
But it suffices to decompose each cyclic permutation of length $m_{k,l}$ above as the product of 
$m_{k,l}-1$ transpositions, which indeed gives 
$$\sum_{l=1}^{s_k}(m_{k,l}-1)=d-s_k=R_k$$
transpositions making the job.
\end{proof}

After applying Lemma (\ref{lem:FibreCritique}) to each critical fiber of $\phi$ outside the polar locus, 
we can now assume that all fibers $\phi^{-1}(c_k)$ are simple branching for $k>n$. By this way, 
we have maximized $B$ without touching at fibers over poles of $\nabla_0$ so far, so $N$ has not changed.
We now discuss how to simplify fibres over poles of $\nabla$ by putting some of their branch points out,
in additional simple branching fibers, without increasing $N-B$; this will however increase $B$.
We do this in successive lemmae discussing on the type of poles.

\begin{lem}\label{lem:FibreNonPeriodic}
Let $0\le k\le n$, i.e. $c_k=p_k$ is a pole of $\nabla$, and assume $\kappa_k=0$ 
and $\theta_k\not\in\mathbb Q\setminus\mathbb Z$.
We can deform $\phi\rightsquigarrow\phi'$ over a neighborhood of $c_k$ so that
the single fiber of $\phi$ with total ramification $R_k$ is replaced by a non branching fiber
over $p_k$ (i.e. $R_k'=0$), and $R_k$ simple branching fibers nearby.
We have $N_k'=d=N_k+R_k$ and $B'=B+R_k$.
\end{lem}

\begin{proof}We proceed like in the proof of Lemma \ref{lem:FibreCritique} by replacing 
the monodromy of $\phi^{-1}(p_k)$ by the product of the identity and $R_k$ transpositions.
Here, the identity stands for the trivial monodromy of the non branching fiber over $p_k$.
\end{proof}

\begin{lem}\label{lem:FibreIrregUnram}
Let $0\le k\le n$, i.e. $c_k=p_k$ is a pole of $\nabla$, and assume $\kappa_k\in\mathbb Z_{>0}$.
We can deform $\phi\rightsquigarrow\phi'$ over a neighborhood of $c_k$ so that
the single fiber of $\phi$ with total ramification $R_k$ is replaced by a non branching fiber
over $p_k$ (i.e. $R_k'=0$), and $R_k$ simple branching fibers nearby.
We have $N_k'=d(\kappa_k+1)=N_k+R_k$ and $B'=B+R_k$.
\end{lem}

The proof is the same as before.

\begin{lem}\label{lem:FibrePeriodic}
Let $0\le k\le n$, i.e. $c_k=p_k$ is a pole of $\nabla$, and assume $\kappa_k=0$
and $\theta_k$ has order $m>1$ modulo $\mathbb Z$.
We can deform $\phi\rightsquigarrow\phi'$ over a neighborhood of $p_k$ so that
the single fiber of $\phi$ with total ramification $R_k$ is replaced by fiber with passport
$$\left[\begin{matrix}m_{k,1}'\\  \vdots\\ m_{k,s_k'}'\end{matrix}\right]=\left[\begin{matrix}m\\ \vdots\\ m\\ 1\\ \vdots\\ 1\end{matrix}\right]$$
and only simple branching fibers nearby.
We have $N_k'\ge N_k$, $B'\ge B$ and $N_k'-B'\le N_k-B$.
\end{lem}

\begin{proof}Here, we cannot just proceed as before. Indeed, if $m$ divides $m_{k,l}$, then there
is not pole on the preimage (or an apparent one that disappears after normalization); however, replacing 
by $m_{k,l}$ non branching points would increase $N_k$ by $m_{k,l}$, but increase $B$ only by $m_{k,l}-1$
so that $N-B$ increases by $1$. We thus have to take care of those points with $m$ dividing $m_{k,l}$.

Consider the euclidean division $m_{k,l}=s^0\cdot m+s^1$. Then, we can replace the point with 
multiplicity $m_{k,l}$ by 
\begin{itemize}
\item $s^0$ points of multiplicity $m$ in the fiber $\phi^{-1}(p_k)$ (contributing to no pole),
\item $s^1$ non branching points in the fiber $\phi^{-1}(p_k)$ (contributing to $s^1$ poles),
\item and $s^0+s^1-1$ additional simple branching fibers around.
\end{itemize}
To realize the deformation of $\phi$, we have to realize the corresponding monodromy representation,
which is easy in this case. Indeed, the concatenation of a cyclic permutation (or a cycle inside a permutation)
runs as follows:
$$\underbrace{(1\ldots \mu)}_{\text{over }p_k}=\underbrace{(1\ldots \mu')(\mu'+1\ldots\mu)}_{\text{over }p_k\text{ after concatenation}}\ \ \ \cdot \underbrace{(1,\mu'+1)}_{\text{simple branching fiber}}$$
where $1<\mu'<\mu$. We just have to repeat this procedure $s_0+s_1-1$ times.
By the way, we get $N_k'=N_k+s^1-1$ (or $N_k'=N_k=0$ if $m$ divides $m_{k,l}$)
 and $B'=B+s^0+s^1-1$. We proceed similarly with all $m_{k,l}\not=m,1$ in the fiber.
\end{proof}

\begin{lem}\label{lem:FibreIrregRam}
Let $0\le k\le n$, i.e. $c_k=p_k$ is a pole of $\nabla$, and assume $\kappa_k\in\mathbb Z_{>0}-\frac{1}{2}$.
We can deform $\phi\rightsquigarrow\phi'$ over a neighborhood of $p_k$ so that
the single fiber of $\phi$ with total ramification $R_k$ is replaced by fiber with passport
$$\left[\begin{matrix}m_{k,1}'\\  \vdots\\ m_{k,s_k'}'\end{matrix}\right]=\left[\begin{matrix}2\\ \vdots\\ 2\\ 1\\ \vdots\\ 1\end{matrix}\right]$$
and only simple branching fibers nearby.
We have  $N_k'-B'= N_k-B$.
\end{lem}

The proof is similar to that of Lemma \ref{lem:FibrePeriodic}. We call {\bf scattered admissible covering} 
(with respect to $(C_0,E_0,\nabla_0)$)
an admissible covering $\phi$ satisfying the conclusion for $\phi'$ in Lemmae \ref{lem:FibreCritique} - \ref{lem:FibreIrregRam}. 
We assume from now on that $\phi$ is scattered. We now deform the differential equation $(E_0,\nabla_0)$ on $C_0$
into a logarithmic one $(E_0',\nabla_0')$ without changing the ramified cover, in such a way that $N-B$ does not increase.
This will allow us to conclude with the classification established by the first author \cite{Diarra1} in the logarithmic case.

\begin{lem}\label{lem:IrregPoleLogPole}
Let $p_k$ be an irregular unramified pole of $\nabla$, i.e. $\kappa_k\in\mathbb Z_{>0}$.
We can deform the differential equation $(E_0,\nabla_0)\rightsquigarrow(E_0',\nabla_0')$ over a neighborhood of $p_k$ 
so that the deformed equation is also normalized with $\kappa_k+1$ simple poles instead 
of a single pole of multiplicity $\kappa_k+1$. If $\phi$ is a scattered covering, then $N'=N$ (and $B'=B$).
\end{lem}

\begin{proof}It is similar to the previous proofs: instead of dealing with the monodromy of the covering, 
we use the monodromy of the differential equation. We can write the monodromy $M$ of $\nabla_0$ around $p_k$
as the product of $\kappa_k+1$ non trivial linear transformations:
$$M=M_0\cdot M_1\cdots M_{\kappa_k},\ \ \ M_i\in\SL(\mathbb C),\ \ \ M_i\not=\pm I.$$
By standard arguments {\it \`a la Riemann-Hilbert}, we can first realize these matrices as local monodromy
of a differential equation over a disc with $(\kappa_k+1)$ simple poles in normal form, and total monodromy $M$.
Next, by surgery over the disc, we replace the single pole $p_k$ in $(E_0,\nabla_0)$ by this new differential 
equation, and get $(E_0',\nabla_0')$ with the desired properties.
\end{proof}

\begin{lem}\label{lem:IrregRamPoleLogPole}
Let $p_k$ be an irregular and ramified pole of $\nabla$, i.e. $\kappa_k\in\mathbb Z_{>0}-\frac{1}{2}$.
We can deform the differential equation $(E_0,\nabla_0)\rightsquigarrow(E_0',\nabla_0')$ over a neighborhood of $p_k$ 
so that the deformed equation is also normalized with $\bar\kappa_k+1=\kappa_k+\frac{3}{2}$ simple poles instead 
of a single pole of multiplicity $\bar\kappa_k+1$, one of which is at the critical point $p_k$ for $\phi$, with $\theta_k'=\frac{1}{2}$. If $\phi$ is a scattered covering, then $N'=N$ (and $B'=B$).
\end{lem}

\begin{proof}The proof is similar, except that we need $M=M_0\cdot M_1\cdots M_{\bar\kappa_k}$ with
$\mathrm{trace}(M_0)=0$.
Then we can realize $M_0$ as the monodromy of a logarithmic differential equation with exponent $\frac{1}{2}$.
\end{proof}

\section{Irregular Euler characteristic}\label{sec:IrregEuler}

Following Poincar\'e, we can associate, to a fuchsian differential equation,
an orbifold structure on the base curve $C_0$ (see also \cite{Diarra1})
and get the notion of orbifold Euler characteristic. Here we define an irregular version of it.
Given $(C_0,E_0,\nabla_0)$ in normal form, 
we define the {\bf orbifold order} $\nu_k$ at a logarithmic
singular point as the order of the local monodromy:
\begin{itemize}
\item $\nu_k\in\mathbb Z_{>1}$ is the order of $[\theta_k\ \mod\ \mathbb Z]$ if $\theta_k\in\mathbb Q\setminus\mathbb Z$,
\item $\nu_k=\infty$ if not.
\end{itemize}
We define the {\bf irregular Euler characteristic} of the differential equation as
$$\chi^{\mathrm{irr}}(C_0,E_0,\nabla_0):=2-2g_0-\sum_{k=1}^{n}(1+\kappa_k)+\sum_{\kappa_k=0}\frac{1}{\nu_k}.$$
In the logarithmic case, this notion coincide to the orbifold Euler charateristic $\chi^{\mathrm{irr}}=\chi^{\mathrm{orb}}$.
We note that the two operations of Lemmae \ref{lem:IrregPoleLogPole} and \ref{lem:IrregRamPoleLogPole},
replacing $(E_0,\nabla_0)$ by a logarithmic equation $(E_0',\nabla_0')$, does not change the irregular Euler characteristic.
Moreover, likely as in \cite[Prop. 2.5]{Diarra1}, we have the following characterization:

\begin{prop}\label{prop:chiPositive}
If $\chi^{\mathrm{irr}}(C_0,E_0,\nabla_0)\ge0$, then the Galois group $\mathrm{Gal}(E_0,\nabla_0)$ 
is virtually abelian. In particular, in the irregular case, $\nabla_0$ has only trivial Stokes and the Galois group 
is one-dimensional: it is diagonal, or dihedral.
\end{prop}

\begin{proof}In the logarithmic case, the curve $C_0$ with its orbifold  structure is a finite quotient of 
the sphere or the torus. The differential equation lifts as a differential equation with trivial or abelian 
monodromy group respectively. Since the Galois group is the Zariski closure of the monodromy 
group in the logarithmic case, we get that $\mathrm{Gal}(E_0,\nabla_0)$ is virtually abelian.
In the irregular case, $\chi^{\mathrm{irr}}\ge0$ gives us
$$-\chi^{\mathrm{irr}}=2g_0-2+\underbrace{\sum_{\kappa_k=0}(1-\frac{1}{\nu_k})}_{\ge 0}
+\underbrace{\sum_{\kappa_k>0}(1+\kappa_k)}_{\ge 1+\frac{1}{2}}\le0.$$
We promptly see that $g_0=0$, likely as in the logarithmic case, and we have 
the following possible formal local data up to bundle transformation
$$\begin{pmatrix}\kappa_1\cdots \kappa_n\\ \theta_1\cdots \theta_n\end{pmatrix}=\begin{pmatrix}1\\ 0\end{pmatrix},\ \ \ 
\begin{pmatrix}\frac{1}{2}\\ 0\end{pmatrix}\ \ \ \text{or}\ \ \ \begin{pmatrix}0&\frac{1}{2}\\ \frac{1}{2}&0\end{pmatrix}.$$
In the first case, we assume that we have an unramified irregular singular point, and the inequality
for $\chi^{\mathrm{irr}}$ gives no place for any other singular point; moreover, $\kappa=1$. 
The monodromy around the unique singular point decomposes as
$$M=
\begin{pmatrix}\lambda&0\\ 0&\lambda^{-1}\end{pmatrix}
\begin{pmatrix}1&s\\ 0&1\end{pmatrix}\begin{pmatrix}1&0\\ t&1\end{pmatrix}=\begin{pmatrix}\lambda(1+st)&\lambda s\\ \lambda^{-1}t&\lambda^{-1}\end{pmatrix}$$
which must be trivial, implying $\lambda=e^{i\pi\theta}=1$ and $s=t=0$. This means that $\theta\in\mathbb Z$,
or equivalently $\theta=0$ after bundle transformation, and we have trivial Stokes matrices. 
The local (and therefore global) Galois group is diagonal like the differential equation.

In the second and third cases, the irregular point is ramified and there is a place for a single logarithmic pole with orbifold
order $\nu=2$. The local monodromy decomposes as
$$M=
\begin{pmatrix}0&1\\ -1&0\end{pmatrix}
\begin{pmatrix}1&s\\ 0&1\end{pmatrix}
=\begin{pmatrix}0&1\\ -1&-s\end{pmatrix}$$
which is never the identity. The second case is therefore impossible. In the third case,
$M$ is also the local monodromy at the logarithmic pole: the trace must be zero, $s=0$, implying trivial Stokes again.
\end{proof}

\begin{prop}\label{prop:bounds} Let $(C_0,E_0,\nabla_0)$ be a normalized differential equation
with local formal data $(\kappa_k,\theta_k)_k$. Let $\phi:C\to C_0$ be a degree $d$ ramified cover
and $(E,\nabla)$ be the pull-back equation. Let $T$ be the dimension of the irregular Teichm\"uller 
deformation space for $(C,E,\nabla)$. Let $B$ be the number of critical values of $\phi$ outside 
the poles of $\nabla_0$. Then we have
$$T-B\ge g-1-d\cdot\chi^{\mathrm{irr}}(C_0,E_0,\nabla_0)$$
where $g$ is the genus of $C$.
\end{prop}

\begin{proof}Let us decompose
$$\begin{matrix}
N=N_1+\cdots+N_n\hfill\hfill\\
R=R_1+\cdots+R_n+B
\end{matrix}$$
where $N_k$ is the number of poles of $\nabla$ (counted with multiplicity) along the fiber $\phi^{-1}(p_k)$,
and $R_k$ is the total ramification along $\phi^{-1}(p_k)$. Now we have
$$T=3g-3+N=g-1+\underbrace{2g-2}_{=d(2g_0-2)+R}+N$$
(by Riemann-Hurwitz) which gives
$$T-B=g-1+d(2g_0-2)+\sum_{k=1}^n(N_k+R_k).$$
Let us lower bound $N_k+R_k$ in fonction of the type of pole $p_k$ for $\nabla_0$.
We note that, along scattering in Lemmae \ref{lem:FibreCritique} - \ref{lem:FibreIrregRam},
the value of $N_k+R_k$ can only decrease, so that it is enough to estimate a lower bound 
for a scattered covering $\phi$.

If $\kappa_k=0$ and $\nu_k\in\mathbb Z_{>0}\cup\{\infty\}$ is the orbifold order, 
then we have
$$d=N_k+m\cdot\nu_k\ \ \ \text{and}\ \ \ R_k=m\cdot(\nu_k-1)$$
so that 
$$N_k+R_k=N_k+\frac{d-N_k}{\nu_k}(\nu_k-1)
=d\left(1-\frac{1}{\nu_k}\right)+\frac{N_k}{\nu_k}\ge d\left(1-\frac{1}{\nu_k}\right).$$
If $\kappa_k\in\mathbb Z_{>0}$, then we find (after scattering)
$$N_k=d(1+\kappa_k)\ \ \ \text{and}\ \ \ R_k=0.$$
If $\kappa_k\in\mathbb Z_{>0}-\frac{1}{2}$, then we have
$$d=m_0+2m_1,\ \ \ N_k=m_0(1+\kappa_k+\frac{1}{2})+m_1(1+2\kappa_k)\ \ \ \text{and}\ \ \ R_k=m_1$$
so that $m_1=\frac{d-m_0}{2}$, and after substitution, we find
$$N_k+R_k=d\left(1+\kappa_k\right)+\frac{m_0}{2}\ge d\left(1+\kappa_k\right).$$
After summing for $k=1,\ldots,n$, we find the expected lower bound.
\end{proof}

\begin{cor}Under assumptions of Proposition \ref{prop:bounds}, if $\nabla_0$ is irregular with non trivial Stokes matrices
and $T-B\le0$, then we have $\chi^{\mathrm{irr}}<0$, $g_0=g=0$,  and 
$$d\vert \chi^{\mathrm{irr}}\vert\le 1.$$
\end{cor}

\begin{proof}The inequality $\chi^{\mathrm{irr}}<0$ directly follows from Proposition \ref{prop:chiPositive}
and the fact we are assuming non trivial Stokes matrices. Then Proposition \ref{prop:bounds} gives
$$0\ge T-B\ge g-1+\underbrace{d\vert \chi^{\mathrm{irr}}\vert}_{>0}$$
which implies $g=0$ (and therefore $g_0=0$). Then we deduce the expected inequality.
\end{proof}

\section{Classification of covers}\label{sec:Classifcover}

In this section, we classify pull-back algebraic solutions of irregular Garnier systems
(see section \ref{sec:algsol}). In other words,
we list all differential equations $(C_0,E_0,\nabla_0)$ and ramified coverings $\phi:C\to C_0$
such that, by deforming $\phi\rightsquigarrow\phi_t$, we get a complete isomonodromic deformation $\phi_t^*(E_0,\nabla_0)$.
In fact, we omit classical solutions which will be discussed in section \ref{sec:classicalsol}
and therefore assume $\chi^{\mathrm{irr}}(C_0,E_0,\nabla_0)<0$ (see Proposition \ref{prop:chiPositive}).
Moreover, equations $(E_0,\nabla_0)$ are listed up to bundle transformation;
in particular, we can assume without lack of generality that $(E_0,\nabla_0)$ is a normalized equation.
In the sequel, we use notations of previous sections. In particular, 
$T$ is the dimension of the irregular Teichm\"uller space of the irregular curve
given by $\phi^*(C_0,E_0,\nabla_0)$, and $B$ is the dimension of deformation 
of $\phi$, obtained by moving the critical values of $\phi$ outside the poles of $\nabla_0$.
To get a complete deformation, we need $T\le B$, and we will assume $T>0$ (otherwise there is no deformation).

\begin{prop}[\cite{Diarra1}]\label{Prop:DiarraPullBack}
Assume $(C_0,E_0,\nabla_0)$ is a (normalized) logarithmic connection
with at least one pole $p_k$ having infinite orbifold order $\nu_k=\infty$, and assume 
$\phi:C\to C_0$ is a scattered ramified cover of degree $d\ge 2$. If $T\le B$ and
$\chi^{\mathrm{irr}}(C_0,E_0,\nabla_0)<0$, then $(C_0,E_0,\nabla_0)$ is hypergeometric,
and up to bundle transformation, we are
in the list of table \ref{table:logarithmic}.
\end{prop}

\begin{proof}This proposition is proved in \cite[second table, page 142]{Diarra1}. It can
be proved directly by using the inequality of Lemma \ref{prop:bounds}.
\end{proof}

\begin{table}[h]
\caption{Logarithmic classification}
\begin{center}
\begin{tabular}{|c|c|c|}
\hline 
Local formal data $\theta_k$ & Degree & Covering passport \\
\hline
$( \frac{1}{2}, \frac{1}{3}, \theta)$ & 6 &
$\begin{pmatrix} \left[\begin{matrix}2\\2\\2\end{matrix}\right] & \left[\begin{matrix}3\\3\end{matrix}\right]  &\left[\begin{matrix}1\\ \vdots\\1\end{matrix}\right] & 3\end{pmatrix}$ \\
\hline
$( \frac{1}{2}, \frac{1}{3}, \theta)$ & 4 &
$\begin{pmatrix}  \left[\begin{matrix}2\\2\end{matrix}\right] & \left[\begin{matrix}3\\1\end{matrix}\right] & \left[\begin{matrix}1\\ \vdots\\1\end{matrix}\right] & 2\end{pmatrix}$ \\
\hline$( \frac{1}{2}, \frac{1}{3}, \theta)$ & 3 &
$\begin{pmatrix}  \left[\begin{matrix}2\\1\end{matrix}\right] & \left[\begin{matrix}3\end{matrix}\right] & \left[\begin{matrix}1\\ 1\\1\end{matrix}\right] & 1\end{pmatrix}$ \\
\hline
$( \frac{1}{2}, \frac{1}{4}, \theta)$ & 4 &
$\begin{pmatrix}  \left[\begin{matrix}2\\2\end{matrix}\right] & \left[\begin{matrix}4\end{matrix}\right] & \left[\begin{matrix}1\\ \vdots\\1\end{matrix}\right] & 1\end{pmatrix}$ \\
\hline
$( \frac{1}{2}, \theta_1, \theta_\infty)$ & 2 &
$\begin{pmatrix}  \left[\begin{matrix}2\end{matrix}\right] & \left[\begin{matrix}1\\1\end{matrix}\right] & \left[\begin{matrix}1\\ 1\end{matrix}\right] & 1\end{pmatrix}$ \\
\hline\end{tabular}
\end{center}
\label{table:logarithmic}
\end{table}

By confluence of the poles of the differential equation, we deduce the following:

\begin{prop}\label{prop:classificationIrreg}
Assume $(C_0,E_0,\nabla_0)$ is a (normalized) differential equation
with at least one irregular pole $p_k$ having non trivial Stokes matrices, and assume 
$\phi:C\to C_0$ is a scattered ramified cover of degree $d\ge 2$. If $T\le B$, 
then $(C_0,E_0,\nabla_0)$ is a degenerate hypergeometric equation,
and up to bundle transformation, we are
in the list of table \ref{table:irregular}.
\end{prop}

\begin{table}[h]
\caption{Irregular classification with scattered cover}
\begin{center}
\begin{tabular}{|c|c|c|c|c|}
\hline 
Local formal data  & Degree & Covering passport & Local formal data & Isomonodromy \\
for $(E_0,\nabla_0)$ &&& for $(E,\nabla)$& equation\\
\hline
$\begin{pmatrix}0& \frac{1}{2}\\ \frac{1}{3} & 0\end{pmatrix}$ & 6 &
$\begin{pmatrix}  \left[\begin{matrix}3\\3\end{matrix}\right]  & \left[\begin{matrix}2\\2\\2\end{matrix}\right] &  3\end{pmatrix}$ &
$\begin{pmatrix}1& 1&1\\ 0&0 & 1\end{pmatrix}$& $\mathrm{Gar}^3(1,1,1)$ \\
\hline
$\begin{pmatrix}0& \frac{1}{2}\\ \frac{1}{3} & 0\end{pmatrix}$ & 4 &
$\begin{pmatrix}   \left[\begin{matrix}3\\1\end{matrix}\right] & \left[\begin{matrix}2\\2\end{matrix}\right]  & 2\end{pmatrix}$ &
$\begin{pmatrix}0& 1&1\\ \frac{1}{3} & 0&1\end{pmatrix}$ & $\mathrm{Gar}^2(0,1,1)$\\
\hline
$\begin{pmatrix}0& \frac{1}{2}\\ \frac{1}{3} & 0\end{pmatrix}$ & 3 &
$\begin{pmatrix}   \left[\begin{matrix}3\end{matrix}\right] & \left[\begin{matrix}2\\1\end{matrix}\right]  & 1\end{pmatrix}$ &
$\begin{pmatrix} \frac{1}{2}&1\\ 0&0\end{pmatrix}$ & $P_{III}^{D_7}$\\
\hline
$\begin{pmatrix}0& \frac{1}{2}\\ \frac{1}{4} & 0\end{pmatrix}$ & 4 &
$\begin{pmatrix}   \left[\begin{matrix}4\end{matrix}\right] & \left[\begin{matrix}2\\2\end{matrix}\right] &  1\end{pmatrix}$ &
$\begin{pmatrix}1&1\\  0&1\end{pmatrix}$&$P_{III}^{D_6}$\\
\hline
$\begin{pmatrix}0& \frac{1}{2}\\ \theta & 0\end{pmatrix}$ & 2 &
$\begin{pmatrix}   \left[\begin{matrix}1\\1\end{matrix}\right] & \left[\begin{matrix}2\end{matrix}\right] &  1\end{pmatrix}$ &
$\begin{pmatrix}0&1&0 \\ \theta-1&0 & \theta \end{pmatrix}$& $P_{V}$\\
\hline
$\begin{pmatrix}0& 1\\ \frac{1}{2} & \theta\end{pmatrix}$ & 2 &
$\begin{pmatrix}    \left[\begin{matrix}2\end{matrix}\right] &  \left[\begin{matrix}1\\1\end{matrix}\right] & 1\end{pmatrix}$ &
$\begin{pmatrix}1& 1\\ \theta-1 & \theta\end{pmatrix}$& $P_{III}^{D_6}$\\
\hline
$\begin{pmatrix}\frac{3}{2}\\ 0\end{pmatrix}$ & 2 &
$\begin{pmatrix}    \left[\begin{matrix}2\end{matrix}\right] &   1\end{pmatrix}$ &
$\begin{pmatrix}3\\ 0\end{pmatrix}$& $P_{II}$\\
\hline
\end{tabular}
\end{center}
\label{table:irregular}
\end{table}

\begin{proof}From Lemmae \ref{lem:IrregPoleLogPole} and  \ref{lem:IrregRamPoleLogPole} the irregular
poles can be scattered as several logarithmic poles with all of them having exponent $\theta$ with infinite
orbifold order, except one in the ramified case, having exponent $\frac{1}{2}$. We are led to the list of table
\ref{table:logarithmic}. We then deduce the list of table \ref{table:irregular} by confluence of poles with 
infinite order $\theta$, and possibly one of them $\frac{1}{2}$. In the first four entries of table \ref{table:logarithmic},
we have no other choice than make the two poles with exponent $\frac{1}{2}$ and $\theta$ confluing into
a ramified irregular pole with $\kappa=\frac{1}{2}$. In the last entry however, we have several possible 
confluences, namely $\{\frac{1}{2},\theta_1\}\rightsquigarrow \kappa=\frac{1}{2}$, $\{\theta_1,\theta_2\}\rightsquigarrow \kappa=1$ and $\{\frac{1}{2},\theta_1,\theta_2\}\rightsquigarrow \kappa=\frac{3}{2}$.
\end{proof}

Finally, by confluence of ramification fibers of $\phi$, we complete the list for pull-back solutions:

\begin{prop}\label{prop:classificationIrregConfl}
Assume $(C_0,E_0,\nabla_0)$ is a (normalized) differential equation
with at least one irregular pole $p_k$ having non trivial Stokes matrices, and assume 
$\phi:C\to C_0$ is a non scattered ramified cover of degree $d\ge 2$. If $T\le B$, 
then $(C_0,E_0,\nabla_0)$ is a degenerate hypergeometric equation,
and up to bundle transformation, we are in the list of table \ref{table:irregularConfl}.
\end{prop}

\begin{table}[h]
\caption{Irregular classification with confluent cover}
\begin{center}
\begin{tabular}{|c|c|c|c|c|}
\hline 
Local formal data  & Degree & Covering passport & Local formal data & Isomonodromy \\
for $(E_0,\nabla_0)$ &&& for $(E,\nabla)$& equation\\
\hline
$\begin{pmatrix}0& \frac{1}{2}\\ \frac{1}{3} & 0\end{pmatrix}$ & 6 &
$\begin{pmatrix}  \left[\begin{matrix}3\\3\end{matrix}\right]  & \left[\begin{matrix}4\\2\end{matrix}\right] &  2\end{pmatrix}$ &
$\begin{pmatrix}1&2\\ 0 & 1\end{pmatrix}$& $\mathrm{Gar}^2(1,2)$ \\
\hline
$\begin{pmatrix}0& \frac{1}{2}\\ \frac{1}{3} & 0\end{pmatrix}$ & 6 &
$\begin{pmatrix}  \left[\begin{matrix}3\\3\end{matrix}\right]  & \left[\begin{matrix}6\end{matrix}\right] &  1\end{pmatrix}$ &
$\begin{pmatrix}3\\  1\end{pmatrix}$& $P_{II}$ \\
\hline
$\begin{pmatrix}0& \frac{1}{2}\\ \frac{1}{3} & 0\end{pmatrix}$ & 4 &
$\begin{pmatrix}   \left[\begin{matrix}3\\1\end{matrix}\right] & \left[\begin{matrix}4\end{matrix}\right]  & 1\end{pmatrix}$ &
$\begin{pmatrix}0& 2\\ \frac{1}{3} & 1\end{pmatrix}$ & $P_{IV}$\\
\hline
\end{tabular}
\end{center}
\label{table:irregularConfl}
\end{table}

\begin{proof}We now inverse the scattering process of Lemmae  \ref{lem:FibreCritique} - \ref{lem:FibreIrregRam}.
To do this, we replace simple branching fibers outside the poles of $\nabla_0$ by additional ramifications
over poles. Note that in table \ref{table:irregular}, each entry satisfies $T=B$ so that we cannot add ramifications
over logarithmic poles with finite orbifold order, otherwise $T-B$ becomes $>0$ (see Lemma \ref{lem:FibrePeriodic}).
The only possibility is therefore to add ramifications over irregular poles of $\nabla_0$, or logarithmic poles 
with exponent $\theta$ having infinite orbifold order. Only the first two lines give examples with $T>0$.
\end{proof}

\begin{remark}We observe that the algebraic solution of $P_{II}(0)$ can be constructed from 
two pull-back constructions (see table \ref{table:irregular} last line and table \ref{table:irregularConfl} line $2$).
This comes from the fact that Airy equation $\begin{pmatrix}\frac{3}{2}\\  0\end{pmatrix}$ is itself pull-back
from Kummer equation $\begin{pmatrix}0& \frac{1}{2}\\ \frac{1}{3} & 0\end{pmatrix}$ by a $3$-fold ramified cover.
Similarly, the algebraic solution of $P_{III}^{D_6}$ appears twice in table \ref{table:irregular} for $\theta=0$
due to the fact 
that Kummer equation $\begin{pmatrix}0& 1\\ \frac{1}{2} & \theta\end{pmatrix}$ is the double cover
of $\begin{pmatrix}0& \frac{1}{2}\\ \frac{1}{4} & 0\end{pmatrix}$ in that case.
\end{remark}

\section{Classification of classical solutions in the case $N=2$.}\label{sec:ClassicalN=2}

In section \ref{sec:Classifcover}, we have given a complete classification of algebraic
solutions of irregular Garnier systems whose linear Galois group is $\SL$. Indeed,
we have classified those solutions of type (1) in Corollary \ref{cor:structure}.
It does not make sense to do the same for solutions of type (2) or (3) in Corollary \ref{cor:structure},
since there are infinitely many, for arbitrary large rank $N$. However, for a given rank,
it makes sense to classify, and we do this in this section for the case $N=2$
which is the first open case after Painlev\'e equations. Recall that we have two solutions 
of type (1) in the case $N=2$, one for each formal data:
$$\begin{pmatrix}0&1&1\\ \frac{1}{3}&0&1\end{pmatrix}\ \ \ 
\text{and}\ \ \ \begin{pmatrix}1&2\\ 0&1\end{pmatrix}$$

\subsection{Without apparent singular point}\label{sec:ClassifDihedral5}
They correspond to the type (2) of Corollary \ref{cor:structure} and have Galois group $D_\infty$.

\begin{equation}
\left\{\begin{matrix}
\begin{pmatrix}0&0&0&\frac{1}{2}\\ \frac{1}{2}&\frac{1}{2}&\frac{1}{2}&0\end{pmatrix}\hfill\hfill
\text{infinite discrete family}\\
\begin{pmatrix}0&0&0&1\\ \frac{1}{2}&\frac{1}{2}&\theta_1&\theta_2\end{pmatrix},\ \ \ 
\begin{pmatrix}0&\frac{1}{2}&0&0\\ \frac{1}{2}&0&\theta_1&\theta_2\end{pmatrix}\hfill\hfill
\text{two-parameter families}\\
\begin{pmatrix}0&0&2\\ \frac{1}{2}&\frac{1}{2}&\theta\end{pmatrix},\ \ 
\begin{pmatrix}0&\frac{1}{2}&1\\ \frac{1}{2}&0&\theta\end{pmatrix},\ \ 
\begin{pmatrix}\frac{1}{2}&\frac{1}{2}&0\\ 0&0&\theta\end{pmatrix},\ \ 
\begin{pmatrix}0&\frac{3}{2}&0\\ \frac{1}{2}&0&\theta\end{pmatrix}\ \ \
\text{one-parameter families}\\
\begin{pmatrix}\frac{1}{2}&\frac{3}{2}\\ 0&0\end{pmatrix}\ \ \ \text{and}\ \ \ 
\begin{pmatrix}0&\frac{5}{2}\\ \frac{1}{2}&0\end{pmatrix}\hfill\hfill
\text{sporadic solutions}
\end{matrix}\right.
\end{equation}

These formal data correspond to those for which the dihedral group $D_\infty$ occur
as a Galois group of the linear equation, up to bundle transformation. In order to find this list,
we have to take into account the following constraints:
\begin{itemize}
\item there are $2$ or $4$ poles where the local monodromy (or Galois group) is anti-diagonal,
and the local formal type must be
$$\begin{pmatrix}0\\ \frac{1}{2}\end{pmatrix}
\ \ \ \text{or}\ \ \ \begin{pmatrix}\frac{k}{2}\\ 0\end{pmatrix},\ \ \ \text{with }k\in\mathbb Z_{>0}\text{ odd.}$$
\item other poles are of local formal type
$$\begin{pmatrix}k\\ \theta\end{pmatrix}\ \ \ \text{with }k\in\mathbb Z_{\ge0},\ \theta\in\mathbb C.$$
\end{itemize}
We would like to insist that it is not necessary to consider particular values for $\theta$ as normalized equations 
with differential Galois group $D_\infty$ having poles with diagonal local monodromy occur in family where
each exponant $\theta$ can be deformed arbitrarily. This comes from the fact that the monodromy representation
itself can be deformed as well.

The first entry corresponds to the unique case with $4$ poles having local anti-diagonal monodromy.
It is an irregular version of Picard-Painlev\'e equation (see \cite{MazzoccoPicard,LPU}): 
there are infinitely many algebraic solutions, in bijection with the orbits of $\mathbb Q\times\mathbb Q$
under the standard action of $\SL(\mathbb Z)$. In fact, if $\tilde C\to C\simeq\mathbb P^1$ denotes the elliptic curve
given by the $2$-fold cover ramifying over the $4$ poles of $(E,\nabla)$, then Picard solutions are related with 
torsion points on $\tilde C$ and how they varry when deforming the poles, and the curve $\tilde C$.
Here, the story is the same. Indeed, the locus of $D_\infty$ Galois group in the moduli space is closed algebraic,
and a differential equation in this closed set has trivial Stokes matrices and is characterized by its monodromy
representation.

For all other cases, recall that the differential equation with dihedral Galois group can be determined
by means of exponents $\theta_i$'s, as for its monodromy, once we know the irregular curve. Therefore,
for each $\theta_i$'s, we get exactly one algebraic solution.

\subsection{With apparent singular point}
They correspond to the type (3) of Corollary \ref{cor:structure} and have Galois group $C_\infty$ or $D_\infty$.
However, as noticed in Remark \ref{rem:ApparentDihedral}, algebraic solutions with Galois group $D_\infty$
and apparent singular points always arise as particular cases of more general algebraic solutions with Galois group $D_\infty$
and arbitrary singular points as listed in section \ref{sec:ClassifDihedral5}. It just remains to complete the list
with those algebraic solutions with Galois group $C_\infty$ and (at least one) apparent singular point.

\begin{equation}
\left\{\begin{matrix}
\begin{pmatrix}0&0&0&1\\ 0&\theta_1&\theta_2&\theta_3\end{pmatrix},\ \ \theta_1+\theta_2+\theta_3=0\ \ \ 
\text{two-parameter family}\\
\begin{pmatrix}0&0&2\\ 0&\theta&-\theta\end{pmatrix},\ \ 
\begin{pmatrix}0&1&1\\ 0&\theta&-\theta\end{pmatrix}\hfill\hfill
\text{one-parameter families}\\
\begin{pmatrix}0&3\\ 0&0\end{pmatrix}\hfill\hfill
\text{sporadic solution}
\end{matrix}\right.
\end{equation}

For each value of $\theta_i$'s, there is exactly one normalized equation once the irregular curve is fixed,
and therefore exactly one algebraic solution.
Solutions with $2$ or more apparent singular points arise as particular cases of these ones by specifying $\theta_i$'s.

\section{Explicit Hamiltonians for some irregular Garnier systems}\label{sec:GarnierHamiltonians}
Here we provide the linear differential equation and the Hamiltonians for some particular formal types
(the complete list comes from \cite{Kimura,Kawamuko}). We translate our notations with Kimura and Kawamuko's.

\subsection{$\mathrm{Kim}_4(1,2,2)$}
The general linear differential equation with (non apparent) poles $x=0,1,\infty$
and corresponding formal type $\begin{pmatrix}2&2&1\\ \theta_0&\theta_1&\theta_\infty\end{pmatrix}$ can be normalized into the form\footnote{From Kimura's formulae \cite{Kimura}, set $\eta_i=1$, 
$\varkappa_i=\theta_i$, $\varkappa=\frac{1}{4}\left[ (\theta_0+\theta_1-1)^2-\theta_\infty^2\right]$
and $K_i=H_i$.}
\begin{equation}\label{eq:Lin(1,2,2)}
L(1,2,2):\ \left\{\begin{matrix}u''+f(x)u'+g(x)u=0\ \ \ \hfill\text{with}\\
f(x)=\frac{t_2}{x^2}+\frac{2-\theta_0}{x}+\frac{t_1}{(x-1)^2}+\frac{2-\theta_1}{x-1}-\sum_{k=1,2}\frac{1}{x-q_k}\hfill\hfill\\
g(x)=\frac{(\theta_0+\theta_1-1)^2-\theta_\infty^2}{4x(x-1)}-\frac{t_1H_1}{x(x-1)^2}+\frac{t_2H_2}{x^2(x-1)}+\sum_{k=1,2}\frac{q_k(q_k-1)p_k}{x(x-1)(x-q_k)}
\end{matrix}\right.
\end{equation}
Singular points $x=q_1,q_2$ are apparent if and only if coefficients $H_1,H_2$ are given by
\begin{equation}\label{eq:Ham(1,2,2)}
H(1,2,2):\ \left\{\begin{matrix}
H_1=\hfill-\frac{q_1^2(q_1-1)^2(q_2-1)}{t_1(q_1-q_2)}
\left(p_1^2-\left(\frac{\theta_0}{q_1}-\frac{t_2}{q_1^2}+\frac{\theta_1-1}{q_1-1}-\frac{t_1}{(q_1-1)^2}\right)p_1+\frac{(\theta_0+\theta_1-1)^2-\theta_\infty^2}{4q_1(q_1-1)}\right)\\
\hfill+\frac{(q_1-1)q_2^2(q_2-1)^2}{t_1(q_1-q_2)}
\left(p_2^2-\left(\frac{\theta_0}{q_2}-\frac{t_2}{q_2^2}+\frac{\theta_1-1}{q_2-1}-\frac{t_1}{(q_2-1)^2}\right)p_1+\frac{(\theta_0+\theta_1-1)^2-\theta_\infty^2}{4q_2(q_2-1)}\right)\\
H_2=\hfill-\frac{q_1^2(q_1-1)^2q_2}{t_2(q_1-q_2)}
\left(p_1^2-\left(\frac{\theta_0-1}{q_1}-\frac{t_2}{q_1^2}+\frac{\theta_1}{q_1-1}-\frac{t_1}{(q_1-1)^2}\right)p_1+\frac{(\theta_0+\theta_1-1)^2-\theta_\infty^2}{4q_1(q_1-1)}\right)\\
\hfill+\frac{q_1q_2^2(q_2-1)^2}{t_2(q_1-q_2)}
\left(p_2^2-\left(\frac{\theta_0-1}{q_2}-\frac{t_2}{q_2^2}+\frac{\theta_1}{q_2-1}-\frac{t_1}{(q_2-1)^2}\right)p_1+\frac{(\theta_0+\theta_1-1)^2-\theta_\infty^2}{4q_2(q_2-1)}\right)
\end{matrix}\right.
\end{equation}
A deformation of (\ref{eq:Lin(1,2,2)}) is isomonodromic if, and only if, parameters satisfies Hamiltonian system
(\ref{eq:HamiltonianGarnierSystem}).
\begin{equation}\label{eq:HamiltonianGarnierSystemK}
\frac{dq_j}{dt_i}=\frac{\partial H_i}{\partial p_j}\ \ \ \text{and}\ \ \ \frac{dp_j}{dt_i}=-\frac{\partial H_i}{\partial q_j}\ \ \ 
\forall i,j=1,2.
\end{equation}
To construct the pull-back solution (second line of Table \ref{table:irregular}),
we start with the differential equation  $\frac{d^2u}{d z^2}+\frac{2}{3z}\frac{du}{d z}-\frac{1}{z}u=0$ 
and consider its pull-back 
by the branch cover
$$z=\phi(x)=\frac{t_2^2(2x-4q_1x+q_1^2+q_1)^3}{16q_1^3(q_1+1)^3x^2(x-1)^2}.$$
Comparing with (\ref{eq:Lin(1,2,2)}), we get the first solution of Theorem \ref{thm:3nonclassical}.

\subsection{$\mathrm{Kim}_6(2,3)$}
The general linear differential equation with (non apparent) poles $x=0,\infty$
and corresponding formal type $\begin{pmatrix}2&3\\ \theta_0&\theta_\infty\end{pmatrix}$ can be normalized into the form\footnote{From Kimura's formulae \cite{Kimura}, set $\eta_0=1$, 
$\varkappa_0=\theta_0$, $\varkappa_\infty=\frac{\theta_0+\theta_\infty-1}{4}$
and $K_i=H_i$.}
\begin{equation}\label{eq:Lin(2,3)}
L(2,3):\ \left\{\begin{matrix}u''+f(x)u'+g(x)u=0\ \ \ \hfill\text{with}\\
f(x)=\frac{t_2}{x^2}+\frac{2-\theta_0}{x}-t_1-\frac{x}{2}-\sum_{k=1,2}\frac{1}{x-q_k}\hfill\hfill\\
g(x)=\frac{\theta_0+\theta_\infty-1}{8}-\frac{H_1}{2x}+\frac{t_2H_2}{x^2}+\sum_{k=1,2}\frac{q_k(q_k-1)p_k}{x(x-1)(x-q_k)}
\end{matrix}\right.
\end{equation}
Singular points $x=q_1,q_2$ are apparent if and only if coefficients $H_1,H_2$ are given by
\begin{equation}\label{eq:Ham(2,3)}
H(2,3):\ \left\{\begin{matrix}
H_1=&\hfill\frac{2q_1^2}{q_1-q_2}\left(p_1^2-\left(\frac{\theta_0}{q_1}-\frac{t_2}{q_1^2}+\frac{q_1}{2}+t_1\right)p_1+\frac{\theta_0+\theta_\infty-1}{8}\right)\\
&\hfill-\frac{2q_2^2}{q_1-q_2}\left(p_2^2-\left(\frac{\theta_0}{q_2}-\frac{t_2}{q_2^2}+\frac{q_2}{2}+t_1\right)p_2+\frac{\theta_0+\theta_\infty-1}{8}\right)\\
H_2=&-\frac{q_1^2q_2}{t_2(q_1-q_2)}\left(p_1^2-\left(\frac{\theta_0-1}{q_1}-\frac{t_2}{q_1^2}+\frac{q_1}{2}+t_1\right)p_1+\frac{\theta_0+\theta_\infty-1}{8}\right)\\
&\hfill+\frac{q_1q_2^2}{t_2(q_1-q_2)}\left(p_2^2-\left(\frac{\theta_0-1}{q_2}-\frac{t_2}{q_2^2}+\frac{q_2}{2}+t_1\right)p_2+\frac{\theta_0+\theta_\infty-1}{8}\right)
\end{matrix}\right.
\end{equation}
and isomonodromic deformations are defined by (\ref{eq:HamiltonianGarnierSystemK}).
To construct the pull-back solution (first line of Table \ref{table:irregularConfl}),
we start with the differential equation  $\frac{d^2u}{d z^2}+\frac{2}{3z}\frac{du}{d z}-\frac{1}{z}u=0$ 
and consider its pull-back 
by the branch cover
$$z=\phi(x)=\frac{(3x^2+8t_1x+4q_1t_1+6q_1^2)^3}{6912x^2}.$$
Comparing with (\ref{eq:Lin(2,3)}), we get the second solution of Theorem \ref{thm:3nonclassical}.

\subsection{$\mathrm{Kaw}_4(5/2,3/2)$}
The general linear differential equation with (non apparent) poles $x=0,\infty$
and corresponding formal type $\begin{pmatrix}\frac{3}{2}&\frac{1}{2}\\ 0&0\end{pmatrix}$ can be normalized into the form\footnote{From Kawamuko's formulae \cite{Kawamuko}, set $\lambda_i=q_i$, $\mu_i=p_i$
and $h_i^{04}=H_i$.} $u''=g(x)u$ with
\begin{equation}\label{eq:Lin(5/2,3/2)}
g(x)=\frac{t_2^2}{4x^3}+\frac{H_1}{x^2}+\frac{H_2}{x}+\frac{t_1}{2}+\frac{x}{4}+\sum_{k=1,2}\left(\frac{3}{4(x-q_k)^2}-\frac{p_k}{x-q_k}\right)
\end{equation}
Then set 
$$u_1=q_1+q_2,\ \ \ u_2=q_1q_2,\ \ \ v_1=\frac{q_1+q_2}{2(q_1-q_2)^2}+\frac{p_1q_1-p_2q_2}{q_1-q_2}$$
$$\text{and}\ \ \ v_2=-\frac{1}{(q_1-q_2)^2}-\frac{p_1-p_2}{q_1-q_2}$$
and 
$$\left\{\begin{matrix}
K_1=2u_1v_1^2+4u_2v_1v_2-4v_1-\frac{u_1^2}{2}-t_1u_1+\frac{u_2}{2}+\frac{t_2^2}{2u_2}\\
t_2K_2=-2u_2v_1^2+2u_2^2v_2^2-2u_2v_2+\frac{u_1u_2}{2}+t_1u_2-\frac{t_2^2u_1}{2u_2}
\end{matrix}\right.$$
Deformation of (\ref{eq:Lin(5/2,3/2)}) is isomonodromic if, and only if
$$\frac{du_j}{dt_i}=\frac{\partial K_i}{\partial v_j}\ \ \ \text{and}\ \ \ \frac{dv_j}{dt_i}=-\frac{\partial K_i}{\partial u_j}\ \ \ 
\forall i,j=1,2.$$
The first classical sporadic solution of Theorem \ref{thm:rank2classical}, 
with dihedral linear Galois group, can be constructed by 
pulling back the differential equation $\frac{d^2u}{d z^2}=\left(\frac{1}{z}-\frac{3}{16x^2}\right)u$ by the ramified covering
$$z=\phi(x)=\frac{(x^2+3t_1x-3t_2)^2}{36 x};$$
after normalizing, and comparing with equation (\ref{eq:Lin(5/2,3/2)}), we get the rational solution
$$(t_1,t_2)\mapsto(u_1,u_2,v_1,v_2):=\left(-t_1,t_2,0,\frac{3}{4t_2}\right).$$

\bibliographystyle{amsplain}

\end{document}